\documentclass[10pt,a4paper]{article}
\usepackage[utf8]{inputenc}
\usepackage[T1]{fontenc}

\usepackage{tikz}
\usetikzlibrary{shapes,arrows,calc, fit, decorations}
\usetikzlibrary{decorations.pathreplacing}
\usetikzlibrary{fit}

\usepackage{float}
\usepackage{a4wide}
\usepackage{amssymb}

\usepackage{authblk}
\usepackage{hyperref}

\usepackage{pifont}

\usepackage[makeroom]{cancel}

\usepackage{array}
\newcolumntype{M}[1]{>{\centering\arraybackslash}m{#1}}
\newcolumntype{N}{@{}m{0pt}@{}}

\usepackage{caption}
\usepackage{subcaption} 

\usepackage{amsthm}
\usepackage{amsmath}
\usepackage{mathrsfs}

\usepackage{wrapfig}

\usepackage[]{algorithm2e}

\usepackage{stackengine}

\newtheorem{theorem}{Theorem}

\newtheorem{proposition}[theorem]{Proposition}
\newtheorem{lemma}[theorem]{Lemma}
\newtheorem{prop}[theorem]{Proposition}
\newtheorem{claim}[theorem]{Claim}

\newcounter{configSPf}

\newcounter{configSPfp}
\newcommand{\configSPfp}[1]{\refstepcounter{configSPfp}\label{#1}C_{\theconfigSPfp}}
\newcounter{configSPn}

\newcounter{configSPnp}
\newcommand{\configSPnp}[1]{\refstepcounter{configSPnp}\label{#1}C_{\theconfigSPnp}}
\newcounter{configSPqp}
\newcommand{\configSPqp}[1]{\refstepcounter{configSPqp}\label{#1}C_{\theconfigSPqp}}
\newcounter{ruleSPf}
\newcommand{\ruleSPf}[1]{\refstepcounter{ruleSPf}\label{#1}R_{\theruleSPf}}

\newcounter{ruleSPfp}

\newcounter{ruleSPn}

\newcounter{ruleSPnp}
\newcommand{\ruleSPnp}[1]{\refstepcounter{ruleSPnp}\label{#1}R_{\theruleSPnp}}

\newcounter{ruleSPqp}

\DeclareMathOperator{\mad}{\mathrm{mad}}

\title{The chromatic number of signed graphs with bounded maximum average 
degree\footnote{Funding: This work was partially supported by the grant HOSIGRA funded by the French National Research Agency (ANR, Agence Nationale de la Recherche) under the contract number ANR-17-CE40-0022.}}
\author{Fabien Jacques, Alexandre Pinlou}
\affil{LIRMM, University of Montpellier, CNRS, Montpellier, France \\  \texttt firstname.lastname@lirmm.fr}
\date{}

\begin{document}

\maketitle

\begin{abstract}
A signed graph is a simple graph with two types of edges: positive and negative edges. Switching a vertex $v$ of a signed graph corresponds to changing the type of each edge incident to $v$.

A homomorphism from a signed graph $G$ to another signed graph $H$ is a mapping $\varphi: V(G) \rightarrow V(H)$ such that, after switching some of the vertices of $G$, $\varphi$ maps every edge of $G$ to an edge of $H$ 
of the same type. The chromatic number $\chi_s(G)$ of a signed graph $G$ is the order of a smallest signed graph $H$ such that there is a homomorphism from $G$ to $H$.

The maximum average degree $\mad(G)$ of a graph $G$ is the maximum of the 
average degrees of all the subgraphs of $G$. We denote $\mathcal{M}_k$ the class of signed graphs with maximum average degree less than $k$ and $\mathcal{P}_g$ the class of planar signed graphs of girth at least $g$.

We prove: 
\begin{itemize}
\item $\chi_s(\mathcal{P}_{7}) \le 5$,
\item $\chi_s(\mathcal{M}_{\frac{17}{5}}) \le 10$ which implies $\chi_s(\mathcal{P}_{5}) \le 10$,
\item $\chi_s(\mathcal{M}_{4-\frac{8}{q+3}}) \le q+1$ with $q$ a prime power congruent to 1 modulo 4.
\end{itemize}
\end{abstract}

\section{Introduction}
There exist several notions of colorings of signed graphs which are all natural extensions and generalizations of colorings of simple graphs. It is well-known that a $k$-coloring of a graph is no more than a homomorphism to the complete graph on $k$ vertices. Using the notion of homomorphism 
of signed graphs introduced by Guenin~\cite{gue05} in 2005, we can define 
a corresponding notion of coloring of signed graphs. This has attracted a 
lot of attention since then and the general question of knowing whether every signed graphs in a family admits a homomorphism to some $H$ has been extensively studied. We can for example cite the expansive papers by Naserasr et al.~\cite{nrs15, HomSGU} where they developed many aspects of this notion. 

Coloring planar graphs has become an illustrious problem in the middle of 
the $19^{\rm{th}}$ century thanks to the Four Color Theorem, that states that four colors are enough to color any simple planar graph. Various branches of this topic then arose, one of which being devoted to the coloring of \emph{sparse} planar graphs. A way to measure the sparseness of a planar graph is to consider its girth (i.e. the length of a shortest cycle): the higher the girth is, the sparser the graph is. Signed coloring of sparse planar graphs has been considerably studied in the last decade (see 
e.g.~\cite{bdnpss20, bkkw04, moprs10, nrs15, nr00, OPS16}). 

A way to get results on sparse planar graphs is to consider graphs (not necessarily planar) with bounded maximum average degree since there exists 
a well-known relation that links the maximum average degree and the girth 
of a planar graph (details are given in the next subsection).

In this paper, we consider homomorphisms of signed graphs with bounded maximum average degree. 

We will first give some classical definitions, define signed graphs and homomorphisms in the remainder of this section and the list of target graphs we will use. Section~\ref{sec:results} introduces the results we obtained and put them in the perspective with the known results. The proof techniques are similar for all our results and we present them in Section~\ref{sec:proof-tech}. Sections~\ref{sec:thSP5} to~\ref{sec:thSPq+} are dedicated to the proofs of our results.

\subsection{Definitions and notation}\label{sec:def}

The degree of a vertex $v$ is its number of neighbors and is denoted by $d(v)$. 
We call a vertex of degree $k$ a \emph{$k$-vertex}, a vertex of degree at 
least $k$ a \emph{$k^+$-vertex} and a vertex of degree at most $k$ a \emph{$k^-$-vertex}. 
We denote by $N(v)$ (resp. $N^-(v)$, $N^+(v)$) the set of vertices that are adjacent (resp. adjacent with a negative edge, adjacent with a positive edge) to a vertex $v$. Let $W$ be a set of vertices, $N(W) = \bigcup_{v \in W} N(v)$ (we also define $N^-(W)$ and $N^+(W)$ similarly).
The \emph{order} of a graph $G$ is the cardinality of its vertex set. The 
\emph{girth} of a graph is the length of a shortest cycle. The \textit{maximum average degree} $\mad(G)$ of a graph $G$ is the maximum of the average degree of all the subgraphs of $G$. There exists a well-known relation that links the maximum average degree and the girth of a planar graph: 
\begin{claim}\label{cl:mad-girth}
Every planar graph $G$ of girth at least $g$ has $\mad(G) <\frac{2g}{g-2}$.  
\end{claim}
Let us denote by $\mathcal{P}_g$ (resp. $\mathcal{M}_d$) the class of planar graphs of girth at least $g$ (resp. the class of graphs with maximum average degree less than $d$). Therefore, $\mathcal{P}_3$ corresponds to the class of planar graphs (since $3$ is the smallest size of a cycle).

\subsection{Signed graphs}

A \textit{signed graph} $G = (V, E, s)$ is a simple graph $(V, E)$ with 
two kinds of edges: positive and negative edges. The signature $s: E(G) \rightarrow \{-1, +1\}$ assigns to each edge its sign (we do not allow parallel edges nor loops). Given a signed graph $G = (V, E, s)$, the \textit{underlying graph} of $G$ is the simple graph $(V, E)$.
\textit{Switching} a vertex $v$ of a signed graph corresponds to reversing the signs of all the edges that are incident to $v$. 
Two signed graphs $G$ and $G'$ are \textit{switching equivalent} if it is 
possible to turn $G$ into $G'$ after some number of switches. The \emph{balance} of a closed walk of a signed graph is the parity of its number of negative edges; a closed walk is said to be \textit{balanced} (resp. \textit{unbalanced}) if it has an even (resp. odd) number of negative edges. 

We can note that a switch does not alter the parity of any closed walk since a switch reverses the sign of an even number of edges of a closed walk. Therefore, Zaslavsky~\cite{Z82} showed the following:
\begin{theorem}[Zaslavsky \cite{Z82}]
\label{thm:Z}
Two signed graphs are switching equivalent if and only if they have the same underlying graph and the same set of balanced cycles.
\end{theorem}

\subsection{Homomorphisms of signed graphs}

Given two signed graphs $G$ and $H$, the mapping $\varphi : V(G)\rightarrow V(H)$ is a \textit{homomorphism} if $\varphi$ preserves adjacencies and the balance of closed walks: an edge $uv$ of $G$ maps to an edge $\varphi(u)\varphi(v)$ of $H$ and a closed walk $v_1v_2\dots v_k$ of $G$ maps to a closed walk  $\varphi(v_1)\varphi(v_2)\dots \varphi(v_k)$ of $H$ of the same balance. This can be seen as coloring the graph $G$ by using the vertices of $H$ as colors. We write $G \rightarrow H$ when there exists an homomorphism from $G$ to $H$.  This notion of homomorphism was introduced by Guenin~\cite{gue05} in 2005 and arises naturally from the fact that 
the balance of closed walks is central in the field of signed graphs.\medskip

Let us introduce the following notion of sign-preserving homomorphisms which is central in studying homomorphisms of signed graphs (see Lemma~\ref{lem:BG} in the next section to understand why) and allows us to give an alternate definition to homomorphisms of signed graphs. Given two signed graphs $G$ and $H$, the mapping $\varphi : V(G)\rightarrow V(H)$ is a \textit{sign-preserving homomorphism} (sp-homomorphism) if $\varphi$ preserves adjacencies and the signs of edges: if vertices $1$ and 
$2$ in $H$ are connected with a positive (resp. negative) edge, then every pair of adjacent vertices in $G$ colored with $1$ and $2$ must be connected with a positive (resp. negative) edge. We write $G \xrightarrow{sp} H$ when there exists an sp-homomorphism from $G$ to $H$. Note that an sp-homomorphism is clearly a homomorphism (adjacencies and balances of closed walk are kept). A reader familiar with the notion of homomorphisms of $2$-edge-colored graphs will recognize that it coincides with the notion of sign-preserving homomorphisms of signed graphs.\medskip

We can then alternatively define homomorphism of signed graph as follows: 
$G \rightarrow H$ if and only if there exists a signed graph $G'$ switching equivalent to $G$ such that $G' \xrightarrow{sp} H$. See \cite{HomSGU} 
for a proof of that equivalence. 

The \textit{chromatic number} $\chi_s(G)$ (resp. \textit{sign-preserving chromatic number} $\chi_{sp}(G)$) of a signed graph $G$ is the order of a 
smallest graph $H$ such that $G \rightarrow H$ (resp. $G \xrightarrow{sp} 
H)$. 
The (sign-preserving) chromatic number $\chi_{s/sp}(\mathcal{C})$ of a class of signed graphs $\mathcal{C}$ is the maximum of the (sign-preserving) chromatic numbers of the graphs in the class.  If $G$ admits a (sp-)homomorphism $\varphi$ to $H$, we say that $G$ is \textit{$H$-(sp-)colorable} and that $\varphi$ is an \emph{$H$-(sp-)coloring} of $G$.

\subsection{Target Graphs}

We present in this subsection the target graphs that will be used to prove our results.

Let $G = (V, E, s)$ be a signed graph. The graph $G$ is said to be \textit{antiautomorphic} if it is isomorphic to $(V, E, -s)$. 
The graph $G$
is said to be \textit{$K_n$-transitive} if for every pair of cliques $\{u_1, u_2, \ldots , u_n\}$ and $\{v_1, v_2, \ldots , v_n\}$ in $G$ such that $s(u_i u_j) = s(v_i v_j)$ for all $i \neq j$, there exists an automorphism that maps $u_i$ to $v_i$ for all $i$. 
For $n = 1$, $2$, or $3$, we say that the graph is \textit{vertex-transitive}, \textit{edge-transitive}, or \textit{triangle-transitive}, respectively. 

The graph $G$ has \emph{Property $P_{k, n}$} if for every sequence of $k$ 
distinct vertices $(v_1, v_2, \dots, v_k)$ that induces a clique in $G$ and for every sign vector $(\alpha_1, \alpha_2, ..., \alpha_k) \in \{-1, +1\}^k$ there exist at least $n$ distinct vertices $\{u_1, u_2, ..., u_n\}$ such that $s(v_i u_j) = \alpha_i$ for $1 \leq i \leq k$ and $1 \leq j 
\leq n$.

Let $q$ be a prime power with $q \equiv 1 \pmod 4$. Let $\mathbb{F}_q$ be the finite field of order $q$.
The \textit{signed Paley graph} $SP_q$ has vertex set $V(SP_q) = \mathbb{F}_q$. Two vertices $u$ and $v \in V(SP_q)$, $u \neq v$, are connected with a positive edge if $u - v$ is a square in $\mathbb{F}_q$ and with a negative edge otherwise.

Notice that this definition is consistent since $q \equiv 1 \pmod 4$ ensures that $-1$ is always a square in $\mathbb{F}_q$ and if $u-v$ is a square then $v-u$ is also a square.

\begin{lemma}[\cite{OPS16}]
\label{lem:PSP}
The signed graph $SP_q$ is vertex-transitive, edge-transitive, antiautomorphic and has properties $P_{1, \frac{q-1}{2}}$ and $P_{2, \frac{q-5}{4}}$.
\end{lemma}
 Figure~\ref{fig:SP_9} gives as an example the signed graph $SP_9$ which contains nine vertices and is complete (only positive edges are displayed, non-edges are negative edges).
 
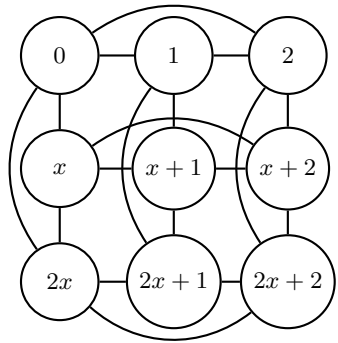
\begin{figure}[!h]
	\centering
	\scalebox{1}
	{
		\begin{tikzpicture}[thick]
		\def \radius {1cm}
		\def \margin {8} 
		\tikzstyle{vertex}=[circle,minimum width=2.9em]
		
		\node[draw, vertex] (0) at (0, 3) {\small$0$};
		\node[draw, vertex] (1) at (1.5, 3) {\small$1$};
		\node[draw, vertex] (2) at (3, 3) {\small$2$};
		\node[draw, vertex] (3) at (0, 1.5) {\small$x$};
		\node[draw, vertex] (4) at (1.5, 1.5) {\small$x+1$};
		\node[draw, vertex] (5) at (3, 1.5) {\small$x+2$};
		\node[draw, vertex] (6) at (0, 0) {\small$2x$};
		\node[draw, vertex] (7) at (1.5, 0) {\small$2x+1$};
		\node[draw, vertex] (8) at (3, 0) {\small$2x+2$};

		\draw (0) -- (1);
		\draw (1) -- (2);
		\draw (2) edge[bend right=35]  (0);
		
		\draw (3) -- (4);
		\draw (4) -- (5);
		\draw (5) edge[bend right=35]  (3);
		
		\draw (6) -- (7);
		\draw (7) -- (8);
		\draw (8) edge[bend left=40]  (6);
		
		\draw (0) -- (3);
		\draw (3) -- (6);
		\draw (6) edge[bend left=35]  (0);
		
		\draw (1) -- (4);
		\draw (4) -- (7);
		\draw (7) edge[bend left=35]  (1);
		
		\draw (2) -- (5);
		\draw (5) -- (8);
		\draw (8) edge[bend left=35]  (2);
		
		\end{tikzpicture}
	}
	\caption{The graph $SP_9$, non-edges are negative edges.}
	\label{fig:SP_9}
\end{figure}

Given a signed graph $G$ of signature $s_G$, we create the \textit{antitwinned graph} of $G$ denoted by $\rho(G)$ as follows:
\begin{itemize}
\item We take two copies $G^{+1}$, $G^{-1}$ of $G$ (the vertex corresponding to $v \in V(G)$ in $G^{i}$ is denoted by~$v^i$)
\item $V(\rho(G)) = V(G^{+1}) \cup V(G^{-1})$
\item $E(\rho(G)) = \{ u^i v^j : uv \in E(G), \ i, j \in \{-1, +1\} \}$
\item $s_{\rho(G)}(u^i v^j) = i \times j \times s_G(u, v)$
\end{itemize}

By construction, for every vertex $v$ of $G$, $v^{-1}$ and $v^{+1}$ are \textit{antitwins}, the positive neighbors of $v^{-1}$ being the negative neighbors of $v^{+1}$ and vice versa.  We say that a signed graph is \textit{antitwinned} if every vertex has a unique antitwin. If $v$ is a vertex in an antitwinned graph, we denote its antitwin with $\overline{v}$.

Antitwinned signed graphs play a central role thanks to the following lemma:
\begin{lemma}[\cite{HomEG}]
\label{lem:BG}
Let $G$ and $H$ be signed graphs. The two following propositions are equivalent:
\begin{itemize}
\item The graph $G$ admits a homomorphism to $H$.
\item The graph $G$ admits a sp-homomorphism to $\rho(H)$.
\end{itemize}
\end{lemma}
In other words, if a signed graph $G=(V,E,s)$ admits an sp-homomorphism 
to an antitwinned target graph on $n$ vertices, then it also admits a homomorphism to a target graph on $\frac{n}{2}$ vertices. We therefore have the following inequalities:

\begin{proposition}[\cite{nrs15}]\label{prop:ineq-switch-2edge}
For every signed graph $G$, we have $\chi_s(G) \le \chi_{sp}(G) \le 2 \cdot \chi_s(G)$.
\end{proposition}

Graphs $\rho(SP_q)$ have the remarkable structural properties given below:
\begin{lemma}[\cite{OPS16}]
\label{lem:PrhoSP}
The graph $\rho(SP_q)$ is vertex-transitive, antiautomorphic and has properties $P_{1, q-1}$, $P_{2, \frac{q-3}{2}}$ and $P_{3, \max(\frac{q-9}{4}, 0)}$.
\end{lemma}

Given a signed graph $G$ which is vertex-transitive, we denote by $G^-$ the graph obtained from $G$ by removing any vertex.
Given a signed graph $G$, we denote by $G^+$ the graph obtained from $G$ by adding a vertex that is connected with a positive edge to every other vertex.  

In the literature, the graph $\rho(SP_q^+)$ is also known as the Tromp-Paley graph $TR(SP_q)$.
This construction improves the properties of $\rho(SP_q)$ at the cost of having only two more vertices (indeed, $|V(\rho(SP_q^+))| = |V(\rho(SP_q))| + 2$). 

\begin{lemma}[\cite{OPS16}]
\label{lem:PTRSP}
The graph $\rho(SP_q^+)$ is vertex-transitive, edge-transitive, antiautomorphic and has properties $P_{1, q}$, $P_{2, \frac{q-1}{2}}$ and $P_{3, \frac{q-5}{4}}$.
\end{lemma}

\section{State of the art and results}\label{sec:results}
As mentioned in the introductory section, the (sign-preserving) chromatic 
number of signed graphs has been studied extensively. Several papers are devoted to planar graphs, planar graphs with given girth, and graphs with 
bounded maximum average degree. 

In 2000, Ne\v set\v ril and Raspaud~\cite{nr00} considered the coloring of $(m,n)$-mixed-graphs (which is a super-class of signed graphs) and they 
proved that $\chi_{sp}(\mathcal{P}_3) \le 80$ by showing that any signed planar graph admits a sp-homomorphism to an antitwinned signed graph on $80$ vertices. This implies as a corollary that $\chi_s(\mathcal{P}_3) \le 
40$ by Lemma~\ref{lem:BG}. The same year, Montejano et al.~\cite{moprs10} 
constructed a signed planar graph $H$ such that $\chi_{sp}(H) = 20$, that implies $\chi_{sp}(\mathcal{P}_3) \ge 20$ and thus $\chi_s(\mathcal{P}_3) \ge 10$. The gap between the lower and upper bounds is huge and in 2020, Bensmail et al.~\cite{conjP} conjectured that $\chi_{sp}(\mathcal{P}_3) = 20$. Recently, Bensmail et al.~\cite{bdnpss20} proved that if this 
conjecture is true, then the target graph is necessarily $\rho(SP_9^+)$. Since this target graph is antitwinned, this would imply that $\chi_s(\mathcal{P}_3) = 10$. This question remains widely open. 

Coloring of sparse (planar) graphs have then been considered. In particular, the following results were obtained:
\begin{description}
\item[Girth $4$:] Ochem et al.~\cite{OPS16} proved that signed planar graphs of girth 4 admit a sp-homomorphism to $\rho(SP_{25})$, that is $\chi_{sp}(\mathcal{P}_4) \le 50$. They also proved that $\chi_{sp}(\mathcal{P}_4) \ge 12$. By Lemma~\ref{lem:BG}, we thus have $6 \le \chi_s(\mathcal{P}_4) \le 25$ since  $\rho(SP_{25})$ is antitwinned. Note that Bensmail et al.~\cite{bdnpss20} conjectured that $\chi_{sp}(\mathcal{P}_4) = 12$ and proved that if this conjecture is true, then the target graph is necessarily $\rho(SP_5^+)$. Since this target graph is antitwinned, this would imply that $\chi_s(\mathcal{P}_4) = 6$. 
\item[Girths $5$, $6$ and $8$ :] Montejano et al.~\cite{moprs10} proved that signed graphs with maximum average degree less than $\frac{10}3$ (resp. $3$, $\frac83$) admit a sp-homomorphism to $\rho(SP_9^+)$ (resp. $\rho(SP_5^+)$, $SP_9^-$), that is $\chi_{sp}(\mathcal{M}_{\frac{10}3}) \le 20$, $\chi_{sp}(\mathcal{M}_{3}) \le 12$ and $\chi_{sp}(\mathcal{M}_{\frac83}) \le 8$. By Claim~\ref{cl:mad-girth}, we get that  $\chi_{sp}(\mathcal{P}_{5}) \le 20$, $\chi_{sp}(\mathcal{P}_{6}) \le 12$ and $\chi_{sp}(\mathcal{P}_{8}) \le 8$. Moreover, since $\rho(SP_9^+)$ and $\rho(SP_5^+)$ are antitwinned, we get that $\chi_s(\mathcal{M}_{\frac{10}3}) \le 10$, $\chi_s(\mathcal{M}_{3}) \le 6$, $\chi_s(\mathcal{P}_{5}) \le 10$, and $\chi_s(\mathcal{P}_{6}) \le 6$ as a corollary by Lemma~\ref{lem:BG}.  Note that since $SP_9^-$ is not antitwinned, Lemma~\ref{lem:BG} does not apply and thus $\chi_s(\mathcal{M}_{\frac83}) \le 6$ and $\chi_s(\mathcal{P}_{8}) \le 6$ are the best known bounds. 
\item[Girth 9:] Charpentier et al.~\cite{cns20} proved that signed graphs with maximum average degree less than $\frac{18}7$ admit a homomorphism to the complete graph on $4$ vertices in which every edge is positive except one that is $\chi_s(\mathcal{M}_{\frac{18}{7}}) \ge 4$ and by Claim~\ref{cl:mad-girth}, $\chi_s(\mathcal{P}_9) \ge 4$. Since an unbalanced cycle of even length has chromatic number $4$~\cite{djmp20}, this bounds are tight. Note that by Lemma~\ref{lem:BG} these results imply $\chi_s(\mathcal{M}_{\frac{18}{7}}) \ge 8$ and $\chi_s(\mathcal{P}_9) \ge 8$ but we can already infer that from the bounds on $\chi_{sp}(\mathcal{M}_{\frac{8}{3}})$ and $\chi_{sp}(\mathcal{P}_8)$.
\item[Girth $g\ge 13$:] Borodin et al.~\cite{bkkw04} proved that for any $g\ge 13$, $\chi_{sp}(\mathcal{P}_g) = 5$.
\end{description}

\begin{table}
  \centering
  \begin{tabular}{|c|c|c|c|c|}
    \hline \hline
    Graph families & $\chi_s$ & $\chi_{sp}$ & Remarks & Refs.\\
    \hline \hline
    $\mathcal{P}_3$ & $10\le\chi_s\le 40$ & $20\le \chi_{sp}\le80$ & & \cite{moprs10,nr00}\\
    \hline
    $\mathcal{P}_4$ & $6\le \chi_s\le25$ & $12\le\chi_{sp}\le50$ & & \cite{OPS16}\\
    \hline
    $\mathcal{M}_{\frac{10}3}$ & $\chi_s\le 10$ & $\chi_{sp}\le 20$ & $\mathcal{P}_5 \subset \mathcal{M}_{\frac{10}3}$ & \cite{moprs10}\\
    \hline
    $\mathcal{M}_{3}$ & $\chi_s\le 6$ & $\chi_{sp}\le 12$ & $\mathcal{P}_6 \subset \mathcal{M}_{3}$ & \cite{moprs10}\\
    \hline
    $\mathcal{M}_{\frac83}$ & $\chi_s\le 6$ & $\chi_{sp}\le 8$ & $\mathcal{P}_8 \subset \mathcal{M}_{\frac83}$ & \cite{moprs10}\\
    \hline
    $\mathcal{M}_{\frac{18}7}$ & $\chi_s = 4$ & $\chi_{sp}\le 8$ & $\mathcal{P}_9 \subset \mathcal{M}_{\frac{18}7}$ & \cite{cns20}\\
    \hline
    $\mathcal{P}_{\ge 13}$ & $\chi_s = 4$ & $\chi_{sp}= 5$ & & \cite{bkkw04}\\
    \hline
    \hline                                    
  \end{tabular}
  \caption{Known results for (sp-)chromatic number of planar graphs with given girth and graphs with bounded maximum average degree.}
  \label{tab:known}
\end{table}

See Table~\ref{tab:known} for a summary.

In the same vein, the first author~\cite{j20} recently studied the chromatic number of signed triangular and hexagonal grids, which are subclasses 
of planar graphs. He respectively proved that 4 (resp. 10) colors are enough for hexagonal (resp. triangular) grids, supporting the conjecture that signed planar graphs have chromatic number at most 10. \medskip

In this paper, we try to find, given a target graph $T$, the highest possible value $m$ such that every graph with maximum average degree less than $m$ admits a homomorphism to $T$. We prove the following three theorems.

\begin{theorem}
\label{thm:SP5}
If a signed graph has maximum average degree smaller than $\frac{20}{7}$, 
it admits a homomorphism to $SP_{5}$. That is $\chi_s(\mathcal{M}_{\frac{20}7}) \le 5$.
\end{theorem}

As a corollary, this gives that $\chi_s(\mathcal{P}_7) \le 5$ and $\chi_{sp}(\mathcal{P}_7) \le 10$, which are new results that contribute to the above-mentioned collection of results.

\begin{theorem}
\label{thm:SP9+}
If a signed graph has maximum average degree smaller than $\frac{17}{5}$, 
it admits a homomorphism to $SP_{9}^+$.  That is $\chi_s(\mathcal{M}_{\frac{17}5}) \le 10$
\end{theorem}

This improves the result of Montejano et al.~\cite{moprs10} saying that $\chi_s(\mathcal{M}_{\frac{10}3}) \le 10$ since $\mathcal{M}_{\frac{10}3} \subset \mathcal{M}_{\frac{17}5}$.
Note that this result contributes to the conjecture that every planar graph admits a homomorphism to $SP_{9}^+$. 

It is not hard to see that signed graphs with maximum average degree at least $4$ have unbounded chromatic number. Consider a complete graph on $n$ vertices $v_1,\dots,v_n$, subdivide each edge $v_iv_j$ by adding a new vertex $u_{ij}$, and for each pair $i,j$, the $2$-path $v_i,u_{ij},v_j$ will have one positive and one negative edge. It easy to see that the average degree of this graph tends to $4$ when $n$ tends to infinity. Moreover, since each pair of $v_i$, $v_j$ is linked by a $2$-path formed by a negative and a positive edge, the $n$ initial vertices must have $n$ distinct colors. Therefore $\chi_{sp}(\mathcal{M}_4)$ is unbounded and thus $\chi_s(\mathcal{M}_4)$ is also unbounded by Proposition~\ref{prop:ineq-switch-2edge}. The following last result gives an upper bound of the chromatic number of signed graphs of maximum average degree $4-\varepsilon$ in function of~$\varepsilon$.
\begin{theorem}
\label{thm:SPq+}
Let $q>9$ be a prime power congruent to 1 modulo 4. If a signed graph has 
maximum average degree smaller than $4 - \frac{8}{q+3}$, it admits a homomorphism to $SP_{q}^+$. That is $\chi_s(\mathcal{M}_{4-\frac8{q+3}}) \le q+1$.
\end{theorem}

\section{Proof techniques}\label{sec:proof-tech}

Since it is easier to manipulate sp-homomorphisms than homomorphisms, we will prove in the remaining four sections that any graph $G\in \mathcal{M}_{\frac{20}7}$ (resp.  $\mathcal{M}_{\frac{17}5}$, $\mathcal{M}_{4-\frac8{q+3}}$) admits an sp-homomorphism to $\rho(SP_5)$ (resp. $\rho(SP_9^+)$, $\rho(SP_{q}^+)$). Theorems~\ref{thm:SP5} to~\ref{thm:SPq+} will then be obtained as corollaries of Lemma~\ref{lem:BG}.

We prove our theorems by contradiction, by assuming that they
have counterexamples. Among all of these counterexamples, we take a
graph $G$ with the fewest number of vertices. Our goal is to prove that
$G$ satisfies structural properties incompatible with having a maximum average degree smaller than a certain value, hence the conclusion.\newline

For each theorem we start by introducing sets of so-called
forbidden configurations, which by minimality $G$ cannot
contain. We then strive to reach a
contradiction with the bounded maximum average degree. To do so, we use the
discharging method. This means that we give some initial weight to
vertices of $G$, then we redistribute those weights, and obtain a
contradiction by double counting the total weight. We present appropriate 
collections of discharging
rules, and argue that every vertex of
$G$ ends up with non-negative weight while the total initial weight
was negative.

Finding a contradiction using a set of forbidden configurations is an
application of a well-known tool called the discharging method. This
approach was introduced more than a century ago to
study the Four-Color Conjecture~\cite{W04}, now a theorem. It is especially
well-suited for studying sparse graphs, and leads to many results, as
shown in two recent surveys~\cite{B13,CW17}. Note that the
discharging method only allows us to reduce the proofs to showing that some configurations are
forbidden. 

In order to prove that $G$ cannot contain a configuration $C_i$, we need to find a suitable smaller graph $G'$ and to extend any
coloring of $G'$ to $G$. Since $G$ is a minimum
counterexample, we get a contradiction if $G$ contains $C_i$.

\section{Proof of Theorem \ref{thm:SP5}}\label{sec:thSP5}
In this section, we prove that any signed graph of maximum average degree 
less than $\frac{20}{7}$ and girth $7$ admit a $\rho(SP_5)$-sp-coloring. To do so, we suppose that this theorem is false and we consider in the remainder of this section a minimal counter-example $G$ w.r.t. its order: it is a smallest signed graph with $\mad(G)<\frac{20}7$ and girth admitting no $\rho(SP_5)$-sp-coloring.

\begin{figure}[H]
	\centering
	\scalebox{0.7}
	{
		\begin{tikzpicture}[thick]
		\def \radius {3cm}
		\def \margin {8} 
		\tikzstyle{vertex}=[circle,minimum width=2.5em]
		
		\node[draw, vertex] (0) at (-72*0+90: 2) {$0$};
		\node[draw, vertex] (1) at (-72*1+90: 2) {$1$};
		\node[draw, vertex] (2) at (-72*2+90: 2) {$2$};
		\node[draw, vertex] (3) at (-72*3+90: 2) {$3$};
		\node[draw, vertex] (4) at (-72*4+90: 2) {$4$};
		
		\draw (0) -- (1);
		\draw (1) -- (2);
		\draw (2) -- (3);
		\draw (3) -- (4);
		\draw (4) -- (0);
		
		\draw[dashed] (0) -- (2);
		\draw[dashed] (1) -- (3);
		\draw[dashed] (2) -- (4);
		\draw[dashed] (3) -- (0);
		\draw[dashed] (4) -- (1);
		
		\end{tikzpicture}
		
	}
	\caption{$SP_5$, the signed Paley graph on 5 vertices.}
	\label{fig:SP5}
\end{figure}
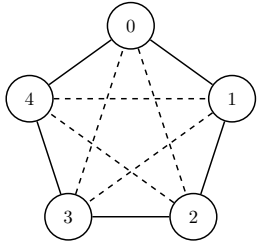

We first introduce some notation in
order to simplify the statements of configurations and rules.

We say that a 3-vertex $v$ is:
\begin{itemize}
\item[$\bullet$] \emph{3-worse} if it has one 2-neighbor (note that by Configuration $C_{\ref{SP5C3}}$ a 3-vertex cannot have more than one 2-neighbor).
\item[$\bullet$] \emph{3-bad} if it has two 3-worse-neighbors (note that by Configuration $C_{\ref{SP5C3}}$ a 3-vertex cannot have three 3-worse-neighbors and by configuration $C_{\ref{SP5C2}}$ a 3-bad cannot be 3-worse).
\item[$\bullet$] \emph{3-good} otherwise.
\end{itemize}

We will say that 3-bad-vertices and 2-vertices are \emph{bad vertices}.

\subsection{Forbidden configurations}
\label{sec:SP5reduction}

We define several configurations $C_0,\ldots,C_{5}$ as follows (see Figure~\ref{fig:SP5configs}).

\begin{itemize}
\item[$\bullet$] $\configSPfp{SP5C0}$ is a 0-vertex.
\item[$\bullet$] $\configSPfp{SP5C1}$ is a 1-vertex.
\item[$\bullet$] $\configSPfp{SP5C2}$ is two adjacent bad vertices.
\item[$\bullet$] $\configSPfp{SP5C3}$ is a 3-vertex with two bad neighbors.
\item[$\bullet$] $\configSPfp{SP5C4}$ is a 3-vertex with one bad neighbor 
adjacent to another 3-vertex with one bad neighbor.
\item[$\bullet$] $\configSPfp{SP5C5}$ is a 4-vertex with three 2-neighbors.
\end{itemize}

Note that every pair of vertices represented in Figure~\ref{fig:SP5configs} is distinct since otherwise $G$ would not have girth at least $7$.

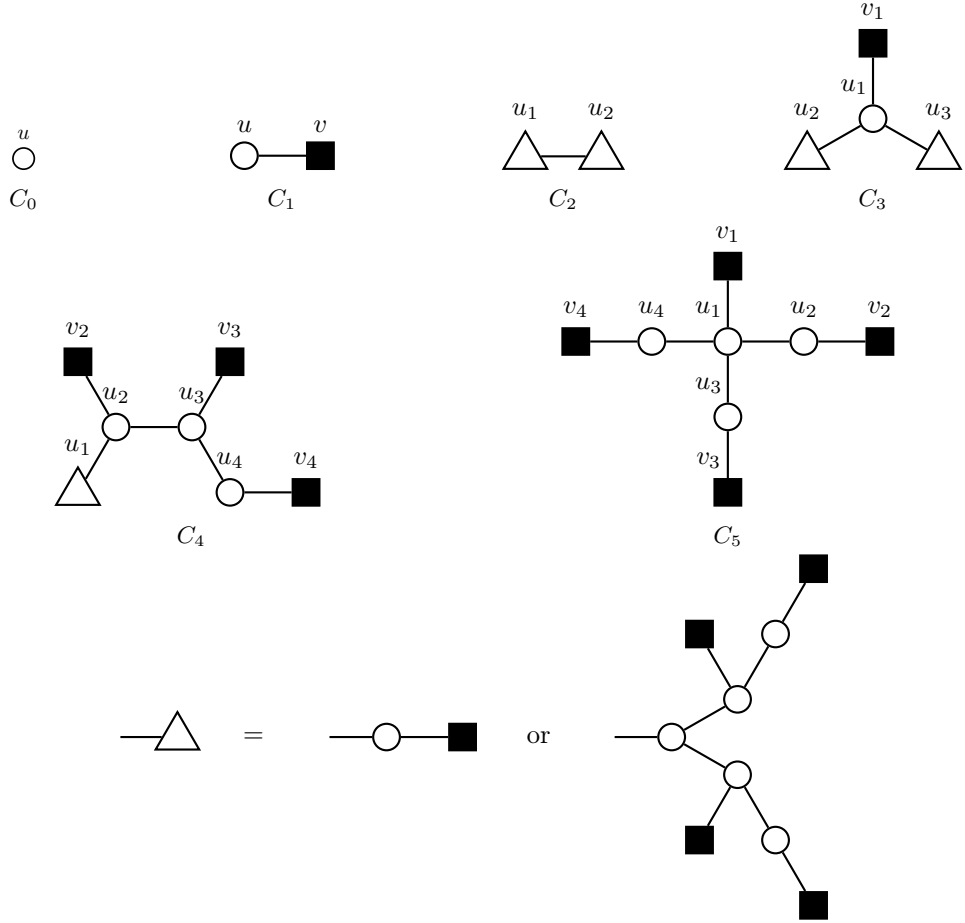
\begin{figure}[H]
    \centering
    \begin{subfigure}[b]{0.2\textwidth}
        \centering 
            \scalebox{0.8}
	        {
                \begin{tikzpicture}[thick]
                    \tikzstyle{vertex}=[circle,minimum width=1em]
                    \tikzstyle{vertex2}=[rectangle, minimum width=1em, minimum height=1em]
        		
        		    \node[draw, vertex] (1) at (0,0) {};
        		    \draw (1.north) node[above]{$u$};
        		    
                \end{tikzpicture} 
            }
        \caption*{$C_0$} 
    \end{subfigure} 
    \begin{subfigure}[b]{0.24\textwidth} 
        \centering 
            \scalebox{1}
	        {
                \begin{tikzpicture}[thick]
                    \tikzstyle{vertex}=[circle,minimum width=1em]
                    \tikzstyle{vertex2}=[rectangle, minimum width=1em, minimum height=1em]
        		
        		    \node[draw, vertex] (1) at (0,0) {};
        		    \node[draw, vertex2, fill] (2) at (1,0) {};
        		    
        		    \draw (1.north) node[above]{$u$};
        		    \draw (2.north) node[above]{$v$};
        		    
        		    \draw (1) -- (2);
                \end{tikzpicture} 
            }
        \caption*{$C_1$} 
    \end{subfigure}
    \begin{subfigure}[b]{0.24\textwidth} 
        \centering 
            \scalebox{1}
	        {
                \begin{tikzpicture}[thick]
                    \tikzstyle{vertex}=[circle,minimum width=1em]
                    \tikzstyle{vertex2}=[rectangle, minimum width=1em, minimum height=1em, fill]
                    \tikzstyle{vertex3}=[regular polygon, regular polygon sides=3, minimum size=1em]
        		
        		    \node[draw, vertex3] (1) at (0,0) {};
        		    \node[draw, vertex3] (2) at (1,0) {};
        		    
        		    \draw (1.north) node[above]{$u_1$};
        		    \draw (2.north) node[above]{$u_2$};
        		    
        		    \draw (1) -- (2);
                \end{tikzpicture} 
            }
        \caption*{$C_2$} 
    \end{subfigure} 
    \begin{subfigure}[b]{0.29\textwidth} 
        \centering 
            \scalebox{1}
	        {
                \begin{tikzpicture}[thick]
                    \tikzstyle{vertex}=[circle,minimum width=1em]
                    \tikzstyle{vertex2}=[rectangle, minimum width=1em, minimum height=1em, fill]
                    \tikzstyle{vertex3}=[regular polygon, regular polygon sides=3, minimum size=1em]
        		
        		    \node[draw, vertex] (0) at (0:0) {};
        		    \node[draw, vertex3] (1) at (0-30:1) {};
        		    \node[draw, vertex2] (3) at (120-30:1) {};
        		    \node[draw, vertex3] (4) at (240-30:1) {};
        		    
        		    \draw (0.north) node[above, shift={(-0.25,0)}]{$u_1$};
        		    \draw (1.north) node[above]{$u_3$};
        		    \draw (3.north) node[above]{$v_1$};
        		    \draw (4.north) node[above]{$u_2$};
        		    
        		    \draw (0) -- (1);
        		    \draw (0) -- (3);
        		    \draw (0) -- (4);
                \end{tikzpicture}  
            }
        \caption*{$C_3$} 
    \end{subfigure}
    
    \centering
    
    \begin{subfigure}[b]{0.44\textwidth} 
        \centering 
            \scalebox{1}
	        {
                \begin{tikzpicture}[thick]
                    \tikzstyle{vertex}=[circle,minimum width=1em]
                    \tikzstyle{vertex2}=[rectangle, minimum width=1em, minimum height=1em, fill]
                    \tikzstyle{vertex3}=[regular polygon, regular polygon sides=3, minimum size=1em]
        		
        		    \node[draw, vertex] (0) at (0:0) {};
        		    \node[draw, vertex2] (1) at (120:1) {};
        		    \node[draw, vertex3] (2) at (-120:1) {};
        		    
        		    \node[draw, vertex, shift={(1,0)}] (0') at (0:0) {};
        		    \node[draw, vertex2, shift={(1,0)}] (1') at (60:1) {};
        		    \node[draw, vertex, shift={(1,0)}] (2') at (-60:1) {};
        		    \node[draw, vertex2, shift={(2,0)}] (3') at (-60:1) {};
        		    
        		    \draw (0.north) node[above]{$u_2$};
        		    \draw (1.north) node[above]{$v_2$};
        		    \draw (2.north) node[above]{$u_1$};
        		    
        		    \draw (0'.north) node[above]{$u_3$};
        		    \draw (1'.north) node[above]{$v_3$};
        		    \draw (2'.north) node[above]{$u_4$};
        		    \draw (3'.north) node[above]{$v_4$};
        		    
        		    \draw (0) -- (1);
        		    \draw (0) -- (2);
        		    
        		    \draw (0) -- (0');
        		    
        		    \draw (0') -- (1');
        		    \draw (0') -- (2');
        		    \draw (3') -- (2');
        		    
                \end{tikzpicture}   
            } 
        \caption*{$C_4$} 
    \end{subfigure} 
    \begin{subfigure}[b]{0.49\textwidth} 
        \centering 
            \scalebox{1}
	        {
                \begin{tikzpicture}[thick]
                    \tikzstyle{vertex}=[circle,minimum width=1em]
                    \tikzstyle{vertex2}=[rectangle, minimum width=1em, minimum height=1em, fill]
        		
        		    \node[draw, vertex] (0) at (0:0) {};
        		    \node[draw, vertex] (1) at (0:1) {};
        		    \node[draw, vertex] (2) at (-90:1) {};
        		    \node[draw, vertex] (3) at (-180:1) {};
        		    \node[draw, vertex2] (1') at (0:2) {};
        		    \node[draw, vertex2] (2') at (-90:2) {};
        		    \node[draw, vertex2] (3') at (-180:2) {};
        		    \node[draw, vertex2] (4) at (-270:1) {};
        		    
        		    \draw (0.north) node[above, shift={(-0.25,0)}]{$u_1$};
        		    \draw (4.north) node[above]{$v_1$};
        		    
        		    \draw (1.north) node[above]{$u_2$};
        		    \draw (1'.north) node[above]{$v_2$};
        		    
        		    \draw (2.north) node[above, shift={(-0.25,0)}]{$u_3$};
        		    \draw (2'.north) node[above, shift={(-0.25,0)}]{$v_3$};
        		    
        		    \draw (3.north) node[above]{$u_4$};
        		    \draw (3'.north) node[above]{$v_4$};
        		    
        		    \draw (0) -- (1);
        		    \draw (0) -- (2);
        		    \draw (0) -- (3);
        		    \draw (0) -- (4);
        		    \draw (1) -- (1');
        		    \draw (2) -- (2');
        		    \draw (3) -- (3');
                \end{tikzpicture}   
            } 
        \caption*{$C_5$} 
    \end{subfigure}
    
    \centering
    
    \begin{subfigure}[b]{0.99\textwidth} 
        \centering 
            \scalebox{1}
	        {
                \begin{tikzpicture}[thick]
                    \tikzstyle{vertex}=[circle,minimum width=1em]
                    \tikzstyle{vertex2}=[rectangle, minimum width=1em, minimum height=1em, fill]
                    \tikzstyle{vertex3}=[regular polygon, regular polygon sides=3, minimum size=1em]
        		
        		    \node[draw, vertex3] (tr) at (0, 0) {};
        		    \draw (tr) -- (-0.75, 0);
        		    
        		    \node (eq) at (1, 0) {$=$};
        		    
        		    \node[draw, vertex] (2) at (2.75, 0) {};
        		    \node[draw, vertex2] (2') at (3.75, 0) {};
        		    \draw (2) -- (-0.75+2.75, 0);
        		    \draw (2) -- (2');
        		    
        		    \node (or) at (4.75, 0) {or};
        		    
        		    \node[draw, vertex, shift={(6,0)}] (b) at (0.5, 0) {};
        		    
        		    \node[draw, vertex, shift={(7.366,-0.5)}] (b0) at (0-90:0) {};
        		    \node[draw, vertex2, shift={(7.366,-0.5)}] (b1) at (-120:1) {};
        		    \node[draw, vertex, shift={(7.366,-0.5)}] (b2) at (120-180:1) {};
        		    \node[draw, vertex2, shift={(7.366,-0.5)}] (b3) at (120-180:2) {};
        		    
        		    \node[draw, vertex, shift={(7.366,0.5)}] (b0') at (0-90:0) {};
        		    \node[draw, vertex2, shift={(7.366,0.5)}] (b1') at (-60+180:1) {};
        		    \node[draw, vertex, shift={(7.366,0.5)}] (b2') at (60:1) {};
        		    \node[draw, vertex2, shift={(7.366,0.5)}] (b3') at (60:2) 
{};

        		    \draw (b0) -- (b1);
        		    \draw (b0) -- (b2);
        		    \draw (b2) -- (b3);
        		    
        		    \draw (b0) -- (b);
        		    \draw (b0') -- (b);
        		    
        		    \draw (b0') -- (b1');
        		    \draw (b0') -- (b2');
        		    \draw (b2') -- (b3');
        		    
        		    \draw (5.75, 0) -- (b);
        		    
                \end{tikzpicture}   
            } 
    \end{subfigure} 
    \caption{Forbidden configurations. Square vertices can be of any degree. White vertices will be removed while proving the non-existence of the 
configuration. Triangles are bad vertices: 2-vertices or 3-bad-vertices.} 

    \label{fig:SP5configs}
\end{figure}

We prove that configurations $C_0$ to $C_5$ are forbidden in $G$. To this 
end, we first prove some
general results we use to prove that the configurations are forbidden. Remember that $\rho(SP_5)$ is vertex-transitive, antiautomorphic and has Properties $P_{1,4}$ and $P_{2, 1}$ by Lemma~\ref{lem:PrhoSP}. \medskip

Consider a signed graph $H$, a vertex $v$ of $H$ of degree $k$, its $k$ neighbors $u_1,u_2,\dots,u_k$. Let $H'=H - v$ and suppose there exists a 
sp-homomorphism $\varphi': H' \xrightarrow{sp} \rho(SP_5)$. With the aim of extending $\varphi'$ to an sp-homomorphism $\varphi$ of the whole graph $H$ we can compute the number of colors forbidden for $v$ by each of its neighbors $u_i$. If we are able to prove that at most $9$ colors are forbidden for $v$, then this means that $\varphi'$ can be extended to an sp-homomorphism $\varphi$ of the whole graph $H$. Note that we may need to recolor some vertices already colored by $\varphi'$. We denote the signature of $H$ by $s_H$. We prove the following claims to this end:

\begin{claim}
\label{c:SP5-2}
2-neighbors forbid one color.
\end{claim}

\begin{claim}
\label{c:SP5-3w}
3-worse-neighbors forbid at most two colors.
\end{claim}

\begin{claim}
\label{c:SP5-3b}
3-bad-neighbors forbid at most one color.
\end{claim}

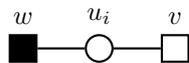
\begin{figure}[H]
    \centering 
        \scalebox{1}
	       {
            \begin{tikzpicture}[thick]
                \tikzstyle{vertex}=[circle,minimum width=1em]
                \tikzstyle{vertex2}=[rectangle, minimum width=1em, minimum height=1em, fill]
                \tikzstyle{vertex3}=[rectangle, minimum width=1em, minimum height=1em]
        	
        	    \node[draw, vertex] (1) at (0,0) {};
        	    \node[draw, vertex3] (2) at (1,0) {};
        	    \node[draw, vertex2] (0) at (-1,0) {};
        	    
        	    \draw (1.north) node[above]{$u_i$};
        	    \draw (2.north) node[above]{$v$};
        	    \draw (0.north) node[above]{$w$};
        	    
        	    \draw (0) -- (1);
        	    \draw (1) -- (2);
            \end{tikzpicture} 
        }
    \caption{Vertex $v$ is adjacent to a 2-vertex.} 
    \label{fig:SP52forbiddenb}
\end{figure} 

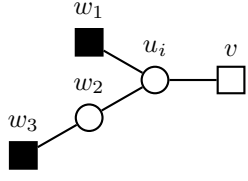
\begin{figure}[H]
    \centering 
        \scalebox{1}
	       {
            \begin{tikzpicture}[thick]
                \tikzstyle{vertex}=[circle,minimum width=1em]
                \tikzstyle{vertex2}=[rectangle, minimum width=1em, minimum height=1em, fill]
                \tikzstyle{vertex3}=[rectangle, minimum width=1em, minimum height=1em]
        	
        	    \node[draw, vertex] (3) at (-1, 0) {};
        	    
        	    \draw (3.north) node[above]{$u_i$};
        	    
        	    \node[draw, vertex2, shift={(-1,0)}] (a1) at (150:1) {};
        	    \node[draw, vertex, shift={(-1,0)}] (a2) at (-150:1) {};
        	    \node[draw, vertex2, shift={(-1,0)}] (a3) at (-150:2) {};
        	    
        	    \draw (a1.north) node[above]{$w_1$};
        	    \draw (a2.north) node[above]{$w_2$};
        	    \draw (a3.north) node[above]{$w_3$};
        	    
        	    \draw (3) -- (a1);
        	    \draw (3) -- (a2);
        	    \draw (a2) -- (a3);
        	    
        	    \node[draw, vertex3] (4) at (0, 0) {};
        	    \draw (4.north) node[above]{$v$};
        	    \draw (3) -- (4);
        	    
            \end{tikzpicture} 
        }
    \caption{Vertex $v$ is adjacent to a 3-worse-vertex.} 
    \label{fig:SP53wforbiddenb}
\end{figure} 

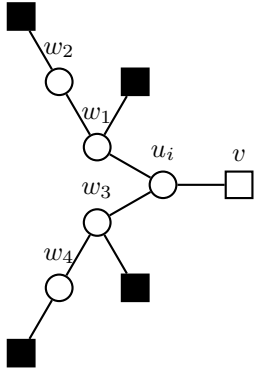
\begin{figure}[H]
    \centering 
        \scalebox{1}
	       {
            \begin{tikzpicture}[thick]
                \tikzstyle{vertex}=[circle,minimum width=1em]
                \tikzstyle{vertex2}=[rectangle, minimum width=1em, minimum height=1em, fill]
                \tikzstyle{vertex3}=[rectangle, minimum width=1em, minimum height=1em]
        	
        	    \node[draw, vertex] (a) at (-0.5, 0) {};
        		    
        		    \node[draw, vertex, shift={(-1.366,-0.5)}] (a0) at (0-90:0) {};
        		    \node[draw, vertex2, shift={(-1.366,-0.5)}] (a1) at (120-180:1) {};
        		    \node[draw, vertex, shift={(-1.366,-0.5)}] (a2) at (-120:1) {};
        		    \node[draw, vertex2, shift={(-1.366,-0.5)}] (a3) at (-120:2) {};
        		    
        		    \node[draw, vertex, shift={(-1.366,0.5)}] (a0') at (0-90:0) {};
        		    \node[draw, vertex2, shift={(-1.366,0.5)}] (a1') at (60:1) {};
        		    \node[draw, vertex, shift={(-1.366,0.5)}] (a2') at (-60+180:1) {};
        		    \node[draw, vertex2, shift={(-1.366,0.5)}] (a3') at (-60+180:2) {};

        		    \draw (a0) -- (a1);
        		    \draw (a0) -- (a2);
        		    \draw (a2) -- (a3);
        		    
        		    \draw (a0) -- (a);
        		    \draw (a0') -- (a);
        		    
        		    \draw (a0') -- (a1');
        		    \draw (a0') -- (a2');
        		    \draw (a2') -- (a3');
        		    
        		    \node[draw, vertex3] (b) at (0.5, 0) {};
        		    
        		    \draw(a) -- (b);
        		    \draw (a.north) node[above]{$u_i$};
        		    \draw (b.north) node[above]{$v$};
        		    \draw (a0.north) node[above]{$w_3$};
        		    \draw (a0'.north) node[above]{$w_1$};
        		     \draw (a2.north) node[above]{$w_4$};
        		      \draw (a2'.north) node[above]{$w_2$};
        	    
            \end{tikzpicture} 
        }
    \caption{Vertex $v$ is adjacent to a 3-bad-vertex.} 
    \label{fig:SP53bforbiddenb}
\end{figure} 

\begin{proof}[Proof of Claim~\ref{c:SP5-2}]
Let $u_i$ be a $2$-neighbor of $v$ and let $w$ be the other neighbor of $u_i$ (see Figure~\ref{fig:SP52forbiddenb}). 
First uncolor vertex $u_i$. Without loss of generality we can suppose that $\varphi'(w)=0$ since $\rho(SP_5)$ is vertex-transitive. If $\varphi(v) \notin \{0, \overline{0}\}$ it is possible to recolor $u_i$ by Property $P_{2, 1}$ of $\rho(SP_5)$. If $s_H(w u_i) = s_H(u_i v)$ and $\varphi(v) = 0$ or $s_H(w u_i) \neq s_H(u_i v)$ and $\varphi(v) = \overline{0}$ then $v$ and $w$ give the same constraints on $u_i$ and we can recolor $u_i$ by Property $P_{1, 4}$ of $\rho(SP_5)$.  Therefore, depending on the signature of the edges $w u_i$ and $u_i v$, the $2$-vertex $u_i$ forbids exactly one color from $v$ and we say that a 2-neighbor forbids one color.
\end{proof}

\begin{proof}[Proof of Claim~\ref{c:SP5-3w}]
Let $u_i$ be a $3$-worse-neighbor of $v$ (see Figure~\ref{fig:SP53wforbiddenb} for vertex naming). 
First uncolor vertices $u_i$ and $w_2$. Without loss of generality we can 
suppose that $\varphi'(w_1)=0$ and that $w_1u_i$ is a positive edge since $\rho(SP_5)$ is vertex-transitive and antiautomorphic. Therefore, $u_i$ may take its color in the set $\{1,\overline{2},\overline{3},4\}$ (i.e the positive neighbors of $0$ in $\rho(SP_5)$). By Claim~\ref{c:SP5-2}, the $2$-vertex $w_2$ forbids one color $f$ from $u_i$. 
\begin{itemize}
\item If $f=1$, then it will always be possible to recolor $u_i$ as long as $\varphi(v)\not\in\{\overline{0},2\}$ (resp. $\varphi(v)\not\in\{0,\overline{2}\}$) if $u_iv$ is positive (resp. negative).
\item If $f=\overline{2}$, then it will always be possible to recolor $u_i$ as long as $\varphi(v)\not\in\{\overline{0},4\}$ (resp. $\varphi(v)\not\in\{0,\overline{4}\}$) if $u_iv$ is positive (resp. negative).
\item If $f=\overline{3}$, then it will always be possible to recolor $u_i$ as long as $\varphi(v)\not\in\{\overline{0},1\}$ (resp. $\varphi(v)\not\in\{0,\overline{1}\}$) if $u_iv$ is positive (resp. negative).
\item If $f=4$, then it will always be possible to recolor $u_i$ as long as $\varphi(v)\not\in\{\overline{0},3\}$ (resp. $\varphi(v)\not\in\{0,\overline{3}\}$) if $u_iv$ is positive (resp. negative).
\item If $f\not\in\{1,\overline{2},\overline{3},4\}$, then it will always 
be possible to recolor $u_i$ as long as $\varphi(v)\neq\overline{0}$ (resp. $\varphi(v)\neq0$) if $u_iv$ is positive (resp. negative).
\end{itemize}

Therefore, 3-worse-neighbors forbid at most two colors.
\end{proof}

\begin{proof}[Proof of Claim~\ref{c:SP5-3b}]
Let $u_i$ be a $3$-bad-neighbor of $v$ (see Figure~\ref{fig:SP53bforbiddenb} for vertex naming). Note that vertices $w_1$ and $w_3$ are $3$-worse vertices. First uncolor vertices $u_i,w_1,\dots,w_4$. By Claim~\ref{c:SP5-3w}, each of $w_1$ and $w_3$ forbids at most 2 colors from $u_i$. Let $F$ be the set of forbidden colors for $u_i$; thus $|F|\le 4$, and let $A=V(\rho(SP_5)) \setminus F$. 
\begin{itemize}
\item If $A$ contains at least four colors that are in the same copy of $SP_5$, it will always be possible to recolor $u_i$ for any color of $v$.
\item If $A$ contains three colors that are in the same copy of $SP_5$ (by symmetry, there are only two cases to consider, either these three colors are consecutive modulo 5 or they are not), there exists only one color 
such that it is not possible to recolor $u_i$.
\end{itemize}

Therefore, 3-bad-vertices can forbid at most one color from their neighbors.
\end{proof}

We now use Claims~\ref{c:SP5-2} to~\ref{c:SP5-3b} to prove that configurations $C_{\ref{SP5C0}}$ to $C_{\ref{SP5C5}}$ cannot appear in $G$. Recall 
that $G$ is a smallest signed graph with $\mad(G) < \frac{20}{7}$ that does not admit a sp-homomorphism to $\rho(SP_5)$.

\begin{lemma}
\label{lem:SP5C0}
The graph $G$ does not contain $C_0$.
\end{lemma}

\begin{proof}
Suppose that $G$ contains $C_0$, a vertex $u$ of degree 0. By minimality of $G$, $G-u$ admits a $\rho(SP_5)$-sp-coloring $\varphi$. Vertex $u$ can 
be mapped to any vertex of $\rho(SP_5)$ to extend $\varphi$ to a $\rho(SP_5)$-sp-coloring of $G$, a contradiction.
\end{proof}

\begin{lemma}
\label{lem:SP5C1}
The graph $G$ does not contain $C_1$.
\end{lemma}

\begin{proof}
Suppose that $G$ contains $C_1$, a vertex $u$ of degree 1. By minimality of $G$, $G-u$ admits a $\rho(SP_5)$-sp-coloring $\varphi$. By Property $P_{1, 4}$ of $\rho(SP_5)$, there are at least $4$ vertices that $u$ can be 
mapped to in order to extend $\varphi$ to a $\rho(SP_5)$-sp-coloring of $G$, a contradiction.
\end{proof}

\begin{lemma}
\label{lem:SP5C2}
The graph $G$ does not contain $C_2$.
\end{lemma}

\begin{proof}
Suppose that $G$ contains $C_2$. By minimality of $G$, $G-\{u_1, u_2\}$ admits a $\rho(SP_5)$-sp-coloring $\varphi$. Since $u_1$ is a bad vertex, it forbids at most one color from $u_2$ by Claims~\ref{c:SP5-2} and \ref{c:SP5-3b}. If $u_2$ is a 2-vertex, by Property $P_{1, 4}$ of $\rho(SP_5)$, it can be colored in at least 3 colors. If $u_2$ is a 3-bad-vertex, its 
neighbors (a bad vertex and two 3-worse-neighbors) forbid at most 4 colors from it by Claims~\ref{c:SP5-2}, \ref{c:SP5-3w} and \ref{c:SP5-3b} so it can be colored in at least 6 colors. It is always possible to extend $\varphi$ to a $\rho(SP_5)$-sp-coloring of $G$, a contradiction.
\end{proof}

\begin{lemma}
\label{lem:SP5C3}
The graph $G$ does not contain $C_3$.
\end{lemma}

\begin{proof}
Suppose that $G$ contains $C_3$. By minimality of $G$, $G-\{u_1, u_2, u_3\}$ admits a $\rho(SP_5)$-sp-coloring $\varphi$. By Claims~\ref{c:SP5-2} and \ref{c:SP5-3b}, $u_2$ and $u_3$ each forbids at most 1 color from $u_1$ and $v_1$ forbids 6 colors from $u_1$ by Property $P_{1, 4}$ of $\rho(SP_5)$. This means that there are at least 2 available colors  for $u_1$. 
It is always possible to extend $\varphi$ to a $\rho(SP_5)$-sp-coloring of $G$, a contradiction.
\end{proof}

\begin{lemma}
\label{lem:SP5C4}
The graph $G$ does not contain $C_4$.
\end{lemma}

\begin{proof}
Suppose that $G$ contains $C_4$. By minimality of $G$, $G-\{u_1, u_2, u_3, u_4\}$ admits a $\rho(SP_5)$-sp-coloring $\varphi$. By vertex-transitivity of $\rho(SP_5)$, we can assume w.l.o.g. that $\varphi(v_2)=0$, and since $\rho(SP_5)$ is antiautomorphic, we can also assume that $s_G(v_2u_2)=+1$. With respect to $v_2$, there are only four available colors for 
$u_2$ which are $A=\{1,\overline{2},\overline{3},4\}$. By Claims~\ref{c:SP5-2} and \ref{c:SP5-3b}, $u_1$ forbids at most 1 color $f$ from $u_2$. 

It is easy to see that $|N^+(A\setminus f)| = |N^-(A\setminus f)| \ge 8$. Therefore, $u_2$ forbids at most $2$ colors from $u_3$. 
Vertex $u_4$ forbids at most one color from $u_3$ by Claim~\ref{c:SP5-2}, 
and $v_3$ forbids 6 colors from $u_3$  by Property $P_{1, 4}$ of $\rho(SP_5)$. Hence $u_3$ can be colored in at least one color. It is always possible to extend $\varphi$ to a $\rho(SP_5)$-sp-coloring of $G$, a contradiction.
\end{proof}

\begin{lemma}
\label{lem:SP5C5}
The graph $G$ does not contain $C_5$.
\end{lemma}

\begin{proof}
Suppose that $G$ contains $C_5$. By minimality of $G$, $G-\{u_1, u_2, u_3, u_4\}$ admits a $\rho(SP_5)$-sp-coloring $\varphi$. By Claims~\ref{c:SP5-2}, $u_2$, $u_3$ and $u_4$ each forbids at most 1 color from $u_1$ and $v_1$ forbids 6 colors from $u_2$ by Property $P_{1, 4}$ of $\rho(SP_5)$. 
This means that there is at least 1 color available for $u_1$. It is always possible to extend $\varphi$ to a $\rho(SP_5)$-sp-coloring of $G$, a contradiction.
\end{proof}

\subsection{Discharging}
\label{sec:SP5positive}

We start by the definition of the initial weighting $\omega$ defined by $\omega(v) = d(v) - \frac{20}{7}$ for each vertex $v$ of degree $d(v)$. By construction, the sum of all the weights $\sum_{v\in V(G)}\omega(v)$ is negative since the maximum average degree of $G$ (and therefore its average degree) is strictly smaller than $\frac{20}{7}$.

We then introduce the following discharging rules:
\begin{itemize}
\item[($\ruleSPf{SP5R1}$)] Every $3^+$-vertex gives $\frac{3}{7}$ to each 
of its 2-neighbors.
\item[($\ruleSPf{SP5R2}$)] Every 3-good, 3-bad or $4^+$-vertex gives $\frac{1}{7}$ to each of its 3-worse-neighbors.
\item[($\ruleSPf{SP5R3}$)] Every 3-good or $4^+$-vertex gives $\frac{1}{7}$ to each of its 3-bad-neighbors.
\end{itemize}

This section is devoted to obtaining a contradiction by proving that
every vertex of $G$ has non-negative final weight after the
discharging procedure. We distinguish several cases for the vertices,
depending on their degree. Remember that $G$ cannot contain configurations $C_0$ to $C_5$ by Lemmas~\ref{lem:SP5C0} to \ref{lem:SP5C5}. Note that since $G$ cannot contain $C_0$ and $C_1$, the
minimum degree of $G$ is~$2$.

\paragraph{$2$-vertices.}
Let $v$ be a 2-vertex. Since $G$ cannot contain $C_{\ref{SP5C2}}$, it does not have any 2-neighbors so it has two 3-worse or $4^+$-neighbors and it receives $\frac{3}{7}$ from each by $R_{\ref{SP5R1}}$. Therefore, the final weight of $v$ is $\omega'(v) = 2 - \frac{20}{7} + 2 \cdot \frac{3}{7} = 0$.

\paragraph{3-worse-vertices.}
Let $v$ be a 3-worse-vertex. Since it is 3-worse, it has one 2-neighbor (but not more since $G$ cannot contain $C_3$) to which it has to give $\frac{3}{7}$. Its other two neighbors are 3-good or $4^+$-vertices (they cannot be 3-worse or 3-bad since $G$ cannot contain $C_3$ and $C_4$) that each gives $\frac{1}{7}$ to it by $R_{\ref{SP5R2}}$. Therefore, the final weight of $v$ is $\omega'(v) = 3 - \frac{20}{7} - \frac{3}{7} + 2 \cdot \frac{1}{7} = 0$.

\paragraph{3-bad-vertices.}
Let $v$ be a 3-bad-vertex. Since it is 3-bad, it has two 3-worse-neighbors (but not more since $G$ cannot contain $C_{\ref{SP5C3}}$) to each of which it has to give $\frac{1}{7}$ by $R_{\ref{SP5R2}}$. Its other neighbor 
is a 3-good or $4^+$-vertex (it cannot be a 2-vertex or a 3-bad-vertex since $G$ cannot contain $C_{\ref{SP5C3}}$ and $C_{\ref{SP5C2}}$) that gives $\frac{1}{7}$ to it by $R_{\ref{SP5R3}}$. Therefore, the final weight of $v$ is $\omega'(v) = 3 - \frac{20}{7} - 2 \cdot \frac{1}{7} + \frac{1}{7} = 0$.

\paragraph{3-good-vertices.}
Let $v$ be a 3-good-vertex. Since it is 3-good, it cannot be 3-bad or 3-worse so it cannot have a 2-neighbor or two 3-worse-neighbor. It can also not have two 3-bad-neighbors since $G$ cannot contain $C_{\ref{SP5C3}}$.

If it has one 3-worse-neighbor, it cannot have a 3-bad-neighbor because $G$ cannot contain $C_{\ref{SP5C4}}$ so it only has to give $\frac{1}{7}$ to the 3-worse-neighbor by $R_{\ref{SP5R2}}$. Therefore, the final weight 
of $v$ is $\omega'(v) = 3 - \frac{20}{7} - \frac{1}{7} = 0$.

If it has one 3-bad-neighbor, it cannot have a 3-worse-neighbor since $G$ 
cannot contain $C_4$ so it only has to give $\frac{1}{7}$ to the 3-bad-neighbor by $R_{\ref{SP5R3}}$. Therefore, the final weight of $v$ is $\omega'(v) = 3 - \frac{20}{7} - \frac{1}{7} = 0$.

\paragraph{$4$-vertices.}
Let $v$ be a 4-vertex. Since $G$ cannot contain $C_5$, it has at most two 
2-neighbors. 

If it has two 2-neighbors and two 3-worse or 3-bad vertices it has final weight $\omega'(v) = 4 - \frac{20}{7} - 2 \cdot \frac{3}{7} - 2 \cdot \frac{1}{7} = 0$ by $R_{\ref{SP5R1}}$, $R_{\ref{SP5R2}}$ and $R_{\ref{SP5R3}}$.

If it has one 2-neighbor and three 3-worse or 3-bad vertices it has final 
weight $\omega'(v) = 4 - \frac{20}{7} - 1 \cdot \frac{3}{7} - 3 \cdot \frac{1}{7} = \frac{2}{7}$ by $R_{\ref{SP5R1}}$, $R_{\ref{SP5R2}}$ and $R_{\ref{SP5R3}}$.

If it has zero 2-neighbors and four 3-worse or 3-bad vertices it has final weight $\omega'(v) = 4 - \frac{20}{7} - 4 \cdot \frac{1}{7} = \frac{4}{7}$ by $R_{\ref{SP5R2}}$ and $R_{\ref{SP5R3}}$.

\paragraph{$5^+$-vertices.}
Let $v$ be an $n$-vertex with $n \geq 5$. In the worst case, $v$ has $n$ 2-neighbors to each of which he has to give $\frac{3}{7}$ by $R_{\ref{SP5R1}}$. Therefore, $v$ has final weight at least $n - \frac{20}{7} - n \cdot \frac{3}{7}$ which is greater than or equal to $0$ for $n \geq 5$.\newline

Every vertex has non-negative weight after discharging so $G$ cannot have 
maximum average degree smaller than $\frac{20}{7}$. This gives us a contradiction and concludes the proof.

\section{Proof of Theorem \ref{thm:SP9+}}
In this section, we prove that any signed graph of maximum average degree 
less than $\frac{17}{5}$ admit a $\rho(SP_9^+)$-sp-coloring. To do so, we 
suppose that this theorem is false and we consider in the remainder of this section a minimal counter-example G w.r.t its order: it is a smallest signed graph with $\mad(G) < \frac{17}{5}$ admitting no $\rho(SP_9^+)$-sp-coloring.

We first introduce some notation in order to simplify the statements of configurations and rules.

We say that a vertex $v$ is \textit{bad} if:
\begin{itemize}
\item $v$ has degree 4 and has one 2-neighbor.
\item $v$ has degree 6 and has four 2-neighbors.
\end{itemize}

If a $4^+$-vertex is not bad, we say that it is $\textit{good}$.

\subsection{Forbidden configurations}

We define several configurations $C_0,\ldots,C_{8}$ as follows (see Figure~\ref{fig:SP9+configs}).

\begin{itemize}

\item[$\bullet$] $\configSPnp{SPnpC0}$ is a 0-vertex.
\item[$\bullet$] $\configSPnp{SPnpC1}$ is a 1-vertex.
\item[$\bullet$] $\configSPnp{SPnpC2}$ is a 2-vertex adjacent to another 2-vertex.
\item[$\bullet$] $\configSPnp{SPnpC3}$ is a 2-vertex adjacent to a 3-vertex.
\item[$\bullet$] $\configSPnp{SPnpC4}$ is a 2-vertex adjacent to two adjacent vertices.
\item[$\bullet$] $\configSPnp{SPnpC5}$ is a 3-vertex.
\item[$\bullet$] $\configSPnp{SPnpC6}$ is a vertex $u$ adjacent to $t$ 2-vertices, $b$ bad vertices and no good vertices with $t + 4 \cdot b < 20$ 
and $b \leq 2$ or $t < 3$ and $b = 3$.
\item[$\bullet$] $\configSPnp{SPnpC7}$ is a vertex $u$ adjacent to $t$ 2-vertices, $b$ bad vertices and one good vertex with $t + 4 \cdot b < 9$.
\item[$\bullet$] $\configSPnp{SPnpC8}$ is a vertex $u$ adjacent to $t$ 2-vertices, 0 bad vertices and two good vertices with $t < 4$.
\end{itemize}

\begin{figure}[H]
    \centering
    \begin{subfigure}[b]{0.05\textwidth}
        \centering 
            \scalebox{1}
	        {
                \begin{tikzpicture}[thick]
                    \tikzstyle{vertex}=[circle,minimum width=1em]
                    \tikzstyle{vertex2}=[rectangle, minimum width=1em, minimum height=1em]
        		
        		    \node[draw, vertex] (1) at (0,0) {};
        		    \draw (1.north) node[above]{$u$};
        		    
                \end{tikzpicture} 
            }
        \caption*{$C_0$} 
    \end{subfigure} 
    \begin{subfigure}[b]{0.10\textwidth} 
        \centering 
            \scalebox{1}
	        {
                \begin{tikzpicture}[thick]
                    \tikzstyle{vertex}=[circle,minimum width=1em]
                    \tikzstyle{vertex2}=[rectangle, minimum width=1em, minimum height=1em]
        		
        		    \node[draw, vertex] (1) at (0,0) {};
        		    \node[draw, vertex2, fill] (2) at (1,0) {};
        		    
        		    \draw (1.north) node[above]{$u$};
        		    \draw (2.north) node[above]{$v$};
        		    
        		    \draw (1) -- (2);
                \end{tikzpicture} 
            }
        \caption*{$C_1$} 
    \end{subfigure} 
    \begin{subfigure}[b]{0.23\textwidth} 
        \centering 
            \scalebox{1}
	        {
                \begin{tikzpicture}[thick]
                    \tikzstyle{vertex}=[circle,minimum width=1em]
                    \tikzstyle{vertex2}=[rectangle, minimum width=1em, minimum height=1em, fill]
        		
        		    \node[draw, vertex] (1) at (0,0) {};
        		    \node[draw, vertex] (2) at (1,0) {};
        		    \node[draw, vertex2] (0) at (-1,0) {};
        		    \node[draw, vertex2] (3) at (2,0) {};
        		    
        		    \draw (0.north) node[above]{$v_1$};
        		    \draw (1.north) node[above]{$u_1$};
        		    \draw (2.north) node[above]{$u_2$};
        		    \draw (3.north) node[above]{$v_2$};
        		    
        		    \draw (0) -- (1);
        		    \draw (1) -- (2);
        		    \draw (2) -- (3);
                \end{tikzpicture} 
            }
        \caption*{$C_2$} 
    \end{subfigure} 
    \begin{subfigure}[b]{0.23\textwidth} 
        \centering 
            \scalebox{1}
	        {
                \begin{tikzpicture}[thick]
                    \tikzstyle{vertex}=[circle,minimum width=1em]
                    \tikzstyle{vertex2}=[rectangle, minimum width=1em, minimum height=1em, fill]
        		
        		    \node[draw, vertex, fill] (0) at (0:0) {};
        		    \node[draw, vertex] (1) at (0:1) {};
        		    \node[draw, vertex2] (2) at (0:2) {};
        		    \node[draw, vertex2] (3) at (120:1) {};
        		    \node[draw, vertex2] (4) at (240:1) {};
        		    
        		    \draw (0.north) node[above]{$u_1$};
        		    \draw (1.north) node[above]{$u_2$};
        		    \draw (2.north) node[above]{$v_3$};
        		    \draw (3.north) node[above]{$v_1$};
        		    \draw (4.north) node[above]{$v_2$};
        		    
        		    \draw (0) -- (1);
        		    \draw (1) -- (2);
        		    \draw (0) -- (3);
        		    \draw (0) -- (4);
                \end{tikzpicture}  
            }
        \caption*{$C_3$} 
    \end{subfigure} 
    \begin{subfigure}[b]{0.16\textwidth} 
        \centering 
            \scalebox{1}
	        {
                \begin{tikzpicture}[thick]
                    \tikzstyle{vertex}=[circle,minimum width=1em]
                    \tikzstyle{vertex2}=[rectangle, minimum width=1em, minimum height=1em, fill]

        		    \node[draw, vertex2] (1) at (0-30:1) {};
        		    \node[draw, vertex] (3) at (120-30:1) {};
        		    \node[draw, vertex2] (4) at (240-30:1) {};
        		    
        		    \draw (1.north) node[above]{$v_2$};
        		    \draw (3.north) node[above]{$u$};
        		    \draw (4.north) node[above]{$v_1$};
        		    
        		    \draw (3) -- (4);
        		    \draw (1) -- (3);
        		    \draw (1) -- (4);
                \end{tikzpicture} 
            }
        \caption*{$C_4$} 
    \end{subfigure} 
     \begin{subfigure}[b]{0.16\textwidth} 
        \centering 
            \scalebox{1}
	        {
                \begin{tikzpicture}[thick]
                    \tikzstyle{vertex}=[circle,minimum width=1em]
                    \tikzstyle{vertex2}=[rectangle, minimum width=1em, minimum height=1em, fill]

        		    \node[draw, vertex] (0) at (0:0) {};
        		    \node[draw, vertex2] (1) at (0-30:1) {};
        		    \node[draw, vertex2] (3) at (120-30:1) {};
        		    \node[draw, vertex2] (4) at (240-30:1) {};
        		    
        		    \draw (0.north) node[above, shift={(-0.25,0)}]{$u$};
        		    \draw (1.north) node[above]{$v_3$};
        		    \draw (3.north) node[above]{$v_1$};
        		    \draw (4.north) node[above]{$v_2$};
        		    
        		    \draw (0) -- (1);
        		    \draw (0) -- (3);
        		    \draw (0) -- (4);
                \end{tikzpicture} 
            }
        \caption*{$C_5$} 
    \end{subfigure} 
    
    \caption{Forbidden configurations $C_{0}$ to $C_{5}$. Square vertices 
can be of any degree. White vertices will be removed to show that the configuration is forbidden.} 
    \label{fig:SP9+configs}
\end{figure}
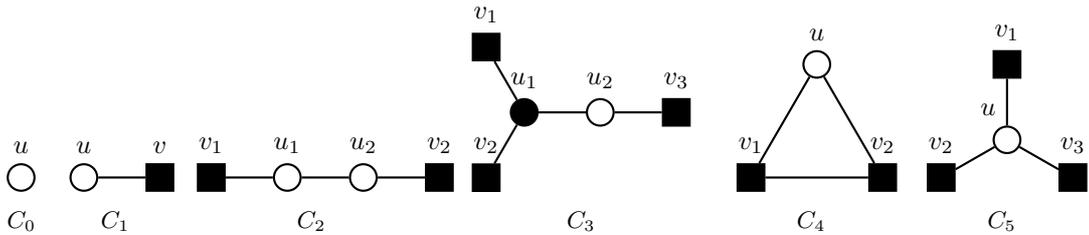

In this section, we prove that configurations $C_{\ref{SPnpC0}}$ to
$C_{\ref{SPnpC8}}$ are forbidden. To this end, we first prove some
generic results we use to prove that the configurations are forbidden. Remember that $\rho(SP_9^+)$ is vertex-transitive, antiautomorphic and has Properties $P_{1,9}$, $P_{2, 4}$ and $P_{3, 1}$ by Lemma~\ref{lem:PrhoSP}.\newline

We say that $G$ is a minimal counter-example if it has the fewest number of $3^+$-vertices and the fewest number of $2^-$-vertices among all the counter-examples that have the same amount of $3^+$-vertices. This will allow us to prove that Configuration~$C_{\ref{SPnpC5}}$ is forbidden.\newline

Given a graph $H$ and a homomorphism from $H$ to $\rho(SP_9^+)$, we say that two vertices of $H$ have the same identity if they are colored with the same color or colors that are antitwins in $\rho(SP_9^+)$. Since $SP_9^+$ has $10$ vertices, there are $10$ different identities in $\rho(SP_9^+)$. If a vertex $v$ is adjacent to $n$ colored vertices with pairwise different identities, these $n$ colors form a clique in $\rho(SP_9^+)$. If $2 \leq n \leq 3$ we can use Property $P_{2, 4}$ or $P_{3,1}$ to color $v$ by Lemma~\ref{lem:PTRSP}.

\begin{lemma}
\label{lem:SPnpC0}\label{TRSP9_C0}
The graph $G$ does not contain Configuration $C_{\ref{SPnpC0}}$.
\end{lemma}

\begin{proof}
Suppose that $G$ contains $C_{\ref{SPnpC0}}$, a vertex $u$ of degree 0. By minimality of $G$, $G-u$ admits a $\rho(SP_9^+)$-coloring $\varphi$. Vertex $u$ can be mapped to any vertex of $\rho(SP_9^+)$ to extend $\varphi$ to a $\rho(SP_9^+)$-coloring of $G$, a contradiction.
\end{proof}

\begin{lemma}\label{TRSP9_C1}
The graph $G$ does not contain Configuration $C_{\ref{SPnpC1}}$.
\end{lemma}

\begin{proof}
Suppose that $G$ contains $C_{\ref{SPnpC1}}$, a vertex $u$ of degree 1. By minimality of $G$, $G-u$ admits a $\rho(SP_9^+)$-coloring $\varphi$. By 
Property $P_{1, 9}$ of $\rho(SP_9^+)$, there are at least $9$ vertices that $u$ can be mapped to in order to extend $\varphi$ to a $\rho(SP_9^+)$-coloring of $G$, a contradiction.
\end{proof}

\begin{lemma}\label{TRSP9_C2}
The graph $G$ does not contain Configuration $C_{\ref{SPnpC2}}$.
\end{lemma}

\begin{proof}
Suppose that $G$ contains $C_{\ref{SPnpC2}}$. By minimality of $G$, $G-\{u_1, u_2\}$ admits a $\rho(SP_9^+)$-coloring $\varphi$. By Property $P_{1, 9}$ of $\rho(SP_9^+)$, there are at least $9$ vertices that $u_1$ can be mapped to in order to extend $\varphi$ to a $\rho(SP_9^+)$-coloring of $G-{u_2}$. One of these vertices (in fact, 8 of them) does not have the same identity as $\varphi(v_2)$. We map $u_1$ to this vertex. By Property $P_{2, 4}$ of $\rho(SP_9^+)$ we can then color $u_2$ since $\varphi(u_1)$ 
and $\varphi(v_2)$ do not share the same identity. We have extended $\varphi$ to a $\rho(SP_9^+)$-coloring of $G$, a contradiction.
\end{proof}

\begin{lemma}\label{TRSP9_C3}
The graph $G$ does not contain Configuration $C_{\ref{SPnpC3}}$.
\end{lemma}

\begin{proof}
Suppose that $G$ contains $C_{\ref{SPnpC3}}$. By minimality of $G$, $G-\{u_2\}$ admits a $\rho(SP_9^+)$-coloring $\varphi$. By Property $P_{2, 4}$ 
of $\rho(SP_9^+)$, there are at least $4$ vertices that $u_1$ can be remapped to (including the one it is already mapped to in $\varphi$). These vertices cannot be antitwins so at least three of them do not have the same identity as $\varphi(v_3)$. We map $u_1$ to one of these three vertices. By Property $P_{2, 4}$ of $\rho(SP_9^+)$ we can then color $u_2$ since $\varphi(u_1)$ and $\varphi(v_3)$ do not share the same identity. We have extended $\varphi$ to a $\rho(SP_9^+)$-coloring of $G$, a contradiction.
\end{proof}

\begin{lemma}
The graph $G$ does not contain  Configuration $C_{\ref{SPnpC4}}$.
\end{lemma}

\begin{proof}
Suppose that $G$ contains $C_{\ref{SPnpC4}}$. By minimality of $G$, $G-\{u\}$ admits a $\rho(SP_9^+)$-coloring $\varphi$. Since $v_1$ and $v_2$ are adjacent, $\varphi(v_1)$ and $\varphi(v_2)$ are also adjacent in $\rho(SP_9^+)$. We can therefore use Property $P_{2, 4}$ of $\rho(SP_9^+)$ to extend $\varphi$ to a $\rho(SP_9^+)$-coloring of $G$, a contradiction.
\end{proof}

\begin{lemma}\label{lem:TRSP9_C4}
The graph $G$ does not contain Configuration $C_{\ref{SPnpC5}}$.
\end{lemma}

\begin{proof}
Suppose that $G$ contains $C_{\ref{SPnpC5}}$. We create a graph $G'$ by removing $u$ from $G$ and adding three 2-vertices $u_1$, $u_2$ and $u_3$ according to Figure~\ref{fig:SPnpC2b} with $s_{G'}(v_1 u_2) = s_{G'}(v_1 
u_3) = s_{G}(v_1 u)$, $s_{G'}(v_2 u_1) = s_{G'}(v_2 u_3) = s_{G}(v_2 u)$ and $s_{G'}(v_3 u_1) = s_{G'}(v_3 u_2) = s_{G}(v_3 u)$. \newline

\begin{figure}[H]
\centering
    \scalebox{1}
	    {
            \begin{tikzpicture}[thick]
                    \tikzstyle{vertex}=[circle,minimum width=1em]
                    \tikzstyle{vertex2}=[rectangle, minimum width=1em, minimum height=1em, fill]

        		    \node[draw, vertex2] (1) at (0-30:1) {};
        		    \node[draw, vertex2] (2) at (120-30:1) {};
        		    \node[draw, vertex2] (3) at (240-30:1) {};
        		    
        		    \node[draw, vertex, fill] (1') at (0+30:0.8) {};
        		    \node[draw, vertex, fill] (2') at (120+30:0.8) {};
        		    \node[draw, vertex, fill] (3') at (240+30:0.8) {};
        		    
        		    \draw (1.north) node[above, shift={(0.18,0)}]{$v_3$};
        		    \draw (3.north) node[above, shift={(-0.18,0)}]{$v_2$};
        		    \draw (2.north) node[above]{$v_1$};
        		    
        		    \draw (1'.north) node[above]{$u_2$};
        		    \draw (3'.north) node[above]{$u_1$};
        		    \draw (2'.north) node[above]{$u_3$};
        		    
        		    \draw (1) -- (1');
        		    \draw (1') -- (2);
        		    \draw (2) -- (2');
        		    \draw (2') -- (3);
        		    \draw (3) -- (3');
        		    \draw (3') -- (1);
                \end{tikzpicture}  
        }
    \caption{The graph $G'$ of Lemma~\ref{lem:TRSP9_C4}. 
    \label{fig:SPnpC2b}}
\end{figure}
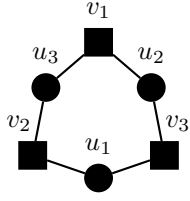

We first prove that $G'$ is smaller than $G$ and that $\mad(G') <  \frac{17}{5}$ in order to prove that $G'$ admits a $\rho(SP_9^+)$-coloring. Using this coloring we then show that $G$ can be colored with $\rho(SP_9^+)$, a contradiction.

Vertices $v_1$, $v_2$ and $v_3$ are $3^+$-vertices in $G$ since Configuration $C_{\ref{SPnpC3}}$ forbids a 2-vertex adjacent to a 3-vertex. Hence, 
$G'$ has less $3^+$-vertices than $G$ so $G'$ is smaller than $G$. \newline

In order tho prove that $\mad(G') <  \frac{17}{5}$, we need to show that a subgraph of $G'$ of maximal average degree has average degree smaller than $\frac{17}{5}$. We use the fact that every subgraph of $G$ has average degree smaller than $\frac{17}{5}$.

Suppose that a subgraph of maximal average degree does not contain $u_1$, 
$u_2$ or $u_3$. The same subgraph in $G$ has the same average degree which is smaller than $\frac{17}{5}$.

Suppose that a subgraph of maximal average degree contains $u_1$ but not $u_2$ or $u_3$. The same subgraph in $G$ with $u$ instead of $u_1$ has the same average degree which is smaller than $\frac{17}{5}$.

Suppose that a subgraph of maximal average degree contains $u_1$ and $u_2$ but not $u_3$. We call this subgraph $H'$. We call $H$ the same subgraph in $G$ with $u$ instead of $u_1$ and $u_2$. The three vertices $v_1$, $v_2$ and $v_3$ must be in this subgraph otherwise we would have at least one vertex of degree 0 or 1 in $H'$ which is not possible since the same subgraph without this vertex would have a greater average degree. Note that we have: $|V(H)| = |V(H')| - 1$ and $|E(H)| = |E(H')| - 1$. The average degree of $H'$ is $\frac{2|E(H')|}{|V(H')|}$. Suppose that this average degree is greater than or equal to $\frac{17}{5}$:

\begin{align*}
\frac{2\cdot|E(H')|}{|V(H')|} &\geq \frac{17}{5}\\
2 \cdot |E(H')| &\geq \frac{17}{5} \cdot |V(H')|\\
2 \cdot |E(H')| - 2 &\geq \frac{17}{5} \cdot |V(H')| -2 \geq  \frac{17}{5} \cdot |V(H')| -1\cdot\frac{17}{5}\\
2 \cdot (|E(H')| - 1) &\geq \frac{17}{5} \cdot (|V(H')| -1)\\
2 \cdot |E(H)| &\geq \frac{17}{5} \cdot |V(H)| \\
\frac{2\cdot|E(H)|}{|V(H)|} &\geq \frac{17}{5}
\end{align*}

We have a contradiction.

We proceed in a similar manner for the case in which a subgraph of $G'$ with maximal average degree contains $u_1$, $u_2$ and $u_3$.\newline

If $\varphi(v_1)$, $\varphi(v_2)$ and $\varphi(v_3)$ all have different identities, we can find a color for $u$ to extend $\varphi$ to $G$ by using Property $P_{3, 1}$ of $\rho(SP_9^+)$.

If there are two vertices $\varphi(v_i)$ and $\varphi(v_j)$ that share the same identity, they can either be colored with  the same color or colors that are antitwins. If they have the same color, we must have $s(v_i u) 
= s(v_j u)$ (because of the way we constructed $G'$) and $v_i$ and $v_j$ induce the same constraints on $u$. If they have colors that are antitwins, we must have $s(v_i u) =-s(v_j u)$ (because of the way we constructed $G'$) and $v_i$ and $v_j$ induce the same constraints on $u$.

We can always extend $\varphi$ to a $\rho(SP_9^+)$-coloring of $G$, a contradiction.
\end{proof}

\begin{lemma}
The graph $G$ does not contain Configurations $C_{\ref{SPnpC6}}$.
\end{lemma}

\begin{proof}
Suppose that $G$ contains Configuration $C_{\ref{SPnpC6}}$. Let $G'=G - 
u$. By minimality of $G$, there exists a homomorphism $\varphi$ from $G'$ 
to $\rho(SP_9^+)$. We want to show that we can extend $\varphi$ into a homomorphism $\varphi'$ from $G$ to $\rho(SP_9^+)$. To do that, we will show that among the $20$ colors that are available (i.e. the $20$ vertices of $\rho(SP_9^+)$), at most $t + 4\cdot b$ are forbidden for $u$ by its bad-neighbors and 2-neighbors if $b \leq 2$ or at most $t + 17$ if $b = 3$. \newline

We introduce the following two propositions that were found using a case analysis on a computer.

\begin{prop}
\label{prop:cf}
Given a set $C$ of $c$ vertices in $\rho(SP_9^+)$ there are at most $f$ vertices that are not positive neighbors (or alternatively negative neighbors) to any of the vertices in $C$:

\begin{tabular}{c|ccccccc}
 $c$ & $0$ & $1$ & $2$ & $3$-$4$ & $5$-$6$ & $7$-$11$ & $12$-$20$ \\ \hline
 $f$ & $20$ & $11$ & $6$ & $4$ & $2$ & $1$ & $0$ 
\end{tabular}
\end{prop}
In other words, if there are $c$ \textit{choices} of colors available for 
a vertex when coloring a graph with $\rho(SP_9^+)$, these $c$ choices \textit{forbid} at most $f$ colors from a neighboring vertex.

\begin{prop}
\label{prop:P24*}
Given a set $C$ of $4$ vertices in $\rho(SP_9^+)$ such that this set can be the result of Property $P_{2, 4}$, there are at most $2$ vertices that 
are not positive neighbors (or alternatively negative neighbors) to any of the vertices in $C$. 

After removing one of the $4$ vertices of $C$, there are at most $3$ vertices that are not positively adjacent (or alternatively negatively adjacent) to any of the $3$ remaining vertices in $C$.
\end{prop}

We now need to prove that each $2$-neighbor of $u$ forbids at most $1$ color from $u$, each bad-neighbor of $u$ forbids at most $4$ colors from $u$ if $b \leq 2$ and three bad neighbors forbid at most $17$ colors from $u$.\newline

\textbf{2-neighbors:} By Property $P_{1, 9}$ of $\rho(SP_9^+)$, a 2-neighbor $v$ of $u$ can be colored in $9$ colors with respect to the color of its neighbor that is not $u$. By Proposition~\ref{prop:cf}, since $v$ can 
be colored in at least $9$ colors, $v$ forbids at most $1$ color from $u$. In other words, $u$ can be colored in at least $19$ colors such that there is at least one of the $9$ colors available for $v$ that is a positive neighbor (or alternatively a negative neighbor) in $\rho(SP_9^+)$ of that color.\newline

\textbf{bad-neighbors:} Note that since Configuration $C_4$ is forbidden, 
a 2-neighbor of $u$ cannot be adjacent to a bad-neighbor of $u$. We consider the following cases:
\begin{itemize}
\item $u$ is adjacent to one bad-vertex $v$:
\begin{itemize}
\item $v$ has degree 4: By Property $P_{2, 4}$ of $\rho(SP_9^+)$ and the fact that a 2-neighbor forbids 1 color, we can use Property~\ref{prop:P24*} to show that $v$ forbids at most $3$ colors from $u$.
\item $v$ has degree 6: By Property $P_{1, 9}$ of $\rho(SP_9^+)$ and the fact that a 2-neighbor forbids 1 color, there are at least $5$ colors available for $v$. By Property~\ref{prop:cf}, $v$ forbids at most $2$ colors 
from $u$.
\end{itemize}
\item $u$ is adjacent to two bad-vertices $v_1$ and $v_2$: Note that since Configuration $C_4$ is forbidden, $v_1$ and $v_2$ cannot be both adjacent and adjacent to the same 2-vertex.
\begin{itemize}
\item $v_1$ and $v_2$ are neither adjacent nor adjacent to a same 2-vertex: For the same reasons as before we know that $v_1$ and $v_2$ each forbid at most $3$ colors.
\item $v_1$ and $v_2$ are adjacent:
\begin{itemize}
\item $v_1$ and $v_2$ are 4-vertices: By Property $P_{1, 9}$ of $\rho(SP_9^+)$ and the fact that a 2-neighbor forbids 1 color, there are at least $8$ colors available for both $v_1$ and $v_2$. A case study by computer reveals that together $v_1$ and $v_2$ forbid at most $2$ colors from $u$.
\item $v_1$ and $v_2$ are 6-vertices: By the fact that a 2-neighbor forbids 1 color, there are at least $16$ colors available for both $v_1$ and $v_2$. This gives us less constraints than the case in which $v_1$ and $v_2$ are 4-vertices. Therefore, together $v_1$ and $v_2$ forbid at most $2$ 
colors from $u$.
\item $v_1$ is a 4-vertex and $v_2$ is a 6-vertex: By Property $P_{1, 9}$ 
of $\rho(SP_9^+)$ and the fact that a 2-neighbor forbids 1 color, there are at least $8$ colors available for $v_1$ and $16$ for $v_2$. A case study by computer reveals that together $v_1$ and $v_2$ forbid at most $2$ colors from $u$.
\end{itemize}
\item $v_1$ and $v_2$ are adjacent to the same 2-vertex:
\begin{itemize}
\item $v_1$ and $v_2$ are 4-vertices: Let $w$ be that 2-vertex, $e_1 = uv_1$, $e_2 = wv_1$, $e_3 = u v_2$ and $e_4 = wv_2$. Suppose that cycle $(u, v_1, w, v_2)$ is balanced. By Theorem~\ref{thm:Z} we can without loss of generality switch a set of vertices such that $s(e_1) = s(e_2)$ and $s(e_3) = s(e_4)$. Therefore, $v_1$ (resp. $v_2$) create the same constraints on both $u$ and $w$ and it suffices to give $w$ the same color as $u$ By Property $P_{2, 4}$ of $\rho(SP_9^+)$ and Property~\ref{prop:P24*} $v_1$ and $v_2$ each forbid at most $2$ colors from $u$. We can thus assume that $(u, v_1, w, v_2)$ is unbalanced. By Theorem~\ref{thm:Z} we can without loss of generality switch a set of vertices such that $s(e_2) = -s(e_4)$. Notice that if we give $v_1$ and $v_2$ different colors 
we can color $w$ by Property $P_{2, 4}$ of $\rho(SP_9^+)$. Let $S_1$ (resp. $S_2$) be the set of possible colors for $v_1$ (resp. $v_2$). By Property $P_{2, 4}$ of $\rho(SP_9^+)$ we know that $|S_1| = |S_2| = 4$. We 
consider the following cases:
\begin{itemize}
\item $|S_1 \cap S_2| = 0$: By Property~\ref{prop:P24*}, $v_1$ and $v_2$ forbid at most $4$ colors from $u$.
\item $|S_1 \cap S_2| = 1$: We remove the common color from $S_1$ such that we can always apply Property $P_{2, 4}$ to color $w$. By Property~\ref{prop:P24*} $v_1$ and $v_2$ forbid at most $5$ colors from $u$.
\item $|S_1 \cap S_2| = 2$: We remove one of the common colors from $S_1$ and the other from $S_2$. By Property~\ref{prop:P24*} $v_1$ and $v_2$ forbid at most $6$ colors from $u$.
\item $|S_1 \cap S_2| = 3$: A case study by computer reveals that this case is not possible.
\item $|S_1 \cap S_2| = 4$: A case study by computer reveals that $v_1$ 
and $v_2$ forbid at most $6$ colors from $u$.
\end{itemize}
\item $v_1$ and $v_2$ are 6-vertices: Suppose that there is only one 2-vertex $w$ adjacent to both $v_1$ and $v_2$. By Property $P_{1, 9}$ of $\rho(SP_9^+)$ and the fact that 2-neighbors forbid at most one color we have 
at least $6$ colors available for $v_1$ and $v_2$. Let $S_1$ and $S_2$ be 
the set of available colors for $v_1$ and $v_2$. Notice that it is not possible for two of the colors in $S_1$ (resp. $S_2$) to be antitwins (they 
would need to be adjacent to a vertex in $\rho(SP_9^+)$ with the same kind of edge). Let us keep only $3$ colors from $S_1$ and $S_2$ such that we 
do not have two colors $c_1 \in S_1$ and $c_2 \in S_2$ such that $c_1 = 
c_2$ or $c_1$ and $c_2$ are antitwins. Choosing any color for $v_1$ and $v_2$ now always allows us to color $w$ by Property $P_{2, 4}$. By Property~\ref{prop:cf}, $v_1$ and $v_2$ each forbid at most $4$ colors from $u$. 
If there are more than one 2-vertex that $v_1$ and $v_2$ are adjacent to we can still apply the same reasoning (except there will be more colors available for $v_1$ and $v_2$).
\item $v_1$ is a 4-vertex and $v_2$ is a 6-vertex: We follow the same reasoning as before and by Property $P_{1, 9}$ and $P_{2, 4}$ of $\rho(SP_9^+)$ and the fact that 2-neighbors forbid at most one color we have at least $4$ colors available for $v_1$ and $6$ for $v_2$. We can guarantee at least 3 distinct colors for $v_1$ and $v_2$ which means by Property~\ref{prop:cf} and \ref{prop:P24*} that they forbid at most $7$ colors from $u$.
\end{itemize}
\end{itemize}
\item $u$ is adjacent to three bad-vertices $v_1$, $v_2$ and $v_3$: By computer, we computed the following properties:

\begin{prop}
\label{prop:19-5-4}
If a vertex $u$ is adjacent to three pairwise adjacent vertices $v_1$, $v_2$ and $v_3$ such that $v_1$ can be colored in $19$ colors, $v_2$ in $5$ 
colors and $v_3$ in $4$ colors then these vertices forbid at most $17$ colors from $u$. 
\end{prop}

\begin{prop}
\label{prop:7-5}
If a vertex $u$ is adjacent to two adjacent vertices $v_1$ and $v_2$ such 
that $v_1$ can be colored in $7$ colors and $v_2$ in $5$ colors then these vertices forbid at most $13$ colors from $u$. 
\end{prop}

\begin{prop}
\label{prop:7-7}
If a vertex $u$ is adjacent to adjacent vertices $v_1$ and $v_2$ such that $v_1$ can be colored in $7$ colors and $v_2$ in the same $7$ colors then these vertices forbid at most $11$ colors from $u$.
\end{prop}

We have 20 cases to consider since $v_1$, $v_2$ and $v_3$ can be of degree 4 or 6 and each pair can either be adjacent, adjacent to the same 2-vertex (or vertices) or neither of those since Configuration $C_4$ is forbidden. The following pictures represent these 20 cases. Note that we do not 
need to consider cases in which a $v_i$ is not adjacent nor adjacent to the same 2-vertex as another of the bad vertices. In such a case we can consider $v_i$ and the other two bad vertices independently using the same reasoning we used when $u$ is adjacent to only one or two bad vertices to 
show that in total $v_1$, $v_2$ and $v_3$ forbid at most 17 colors from $u$. In the picture, the number inside the vertex corresponds to its degree and the number above the vertex corresponds to the number of available colors by Property $P_{1, 9}$ or $P_{2, 4}$ of $\rho(SP_9^+)$ and the fact that 2-neighbors forbid at most one color. Dashed lines represent two vertices that are adjacent to the same 2-vertex. The vertex on the left is 
$v_1$, in the middle $v_2$ and on the right $v_3$. We denote by $S_1$, $S_2$ and $S_3$ the sets of colors available for $v_1$, $v_2$ and $v_3$ respectively. 

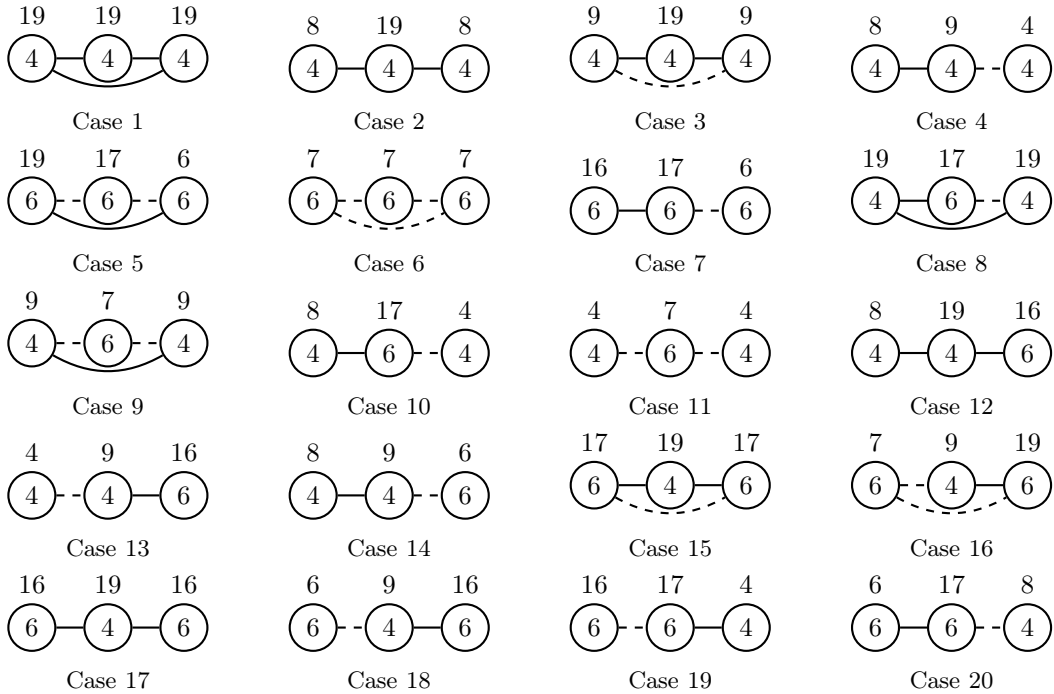
\begin{figure}[H]
    \centering
    \begin{subfigure}[b]{0.24\textwidth}
        \centering 
            \scalebox{1}
	        {
                \begin{tikzpicture}[thick]
                     \tikzstyle{vertex}=[circle,minimum width=1em]
                    \tikzstyle{vertex2}=[rectangle, minimum width=1em, minimum height=1em, fill]
        		
        		    \node[draw, vertex] (0) at (-1,0) {$4$};
        		    \node[draw, vertex] (1) at (0,0) {$4$};
        		    \node[draw, vertex] (2) at (1,0) {$4$};
        		    
        		    \draw (1.north) node[above]{$19$};
        		    \draw (2.north) node[above]{$19$};
        		    \draw (0.north) node[above]{$19$};
        		    
        		    \draw (0) -- (1);
        		    \draw (1) -- (2);
        		    \draw (0) edge[bend right] (2);
        		    
                \end{tikzpicture} 
            }
        \caption*{Case 1} 
    \end{subfigure} 
    \begin{subfigure}[b]{0.24\textwidth} 
        \centering 
            \scalebox{1}
	        {
                \begin{tikzpicture}[thick]
                     \tikzstyle{vertex}=[circle,minimum width=1em]
                    \tikzstyle{vertex2}=[rectangle, minimum width=1em, minimum height=1em, fill]
        		
        		    \node[draw, vertex] (0) at (-1,0) {$4$};
        		    \node[draw, vertex] (1) at (0,0) {$4$};
        		    \node[draw, vertex] (2) at (1,0) {$4$};
        		    
        		    \draw (0.north) node[above]{$8$};
        		    \draw (1.north) node[above]{$19$};
        		    \draw (2.north) node[above]{$8$};
        		    
        		    \draw (0) -- (1);
        		    \draw (1) -- (2);
        		    
                \end{tikzpicture} 
            }
        \caption*{Case 2} 
    \end{subfigure} 
    \begin{subfigure}[b]{0.24\textwidth} 
        \centering 
            \scalebox{1}
	        {
                \begin{tikzpicture}[thick]
                     \tikzstyle{vertex}=[circle,minimum width=1em]
                    \tikzstyle{vertex2}=[rectangle, minimum width=1em, minimum height=1em, fill]
        		
        		    \node[draw, vertex] (0) at (-1,0) {$4$};
        		    \node[draw, vertex] (1) at (0,0) {$4$};
        		    \node[draw, vertex] (2) at (1,0) {$4$};
        		    
        		    \draw (0.north) node[above]{$9$};
        		    \draw (1.north) node[above]{$19$};
        		    \draw (2.north) node[above]{$9$};
        		    
        		    \draw (0) -- (1);
        		    \draw (1) -- (2);
        		    \draw (0) edge[bend right, dashed] (2);
        		    
                \end{tikzpicture} 
            }
        \caption*{Case 3} 
    \end{subfigure} 
    \begin{subfigure}[b]{0.24\textwidth} 
        \centering 
            \scalebox{1}
	       {
                \begin{tikzpicture}[thick]
                     \tikzstyle{vertex}=[circle,minimum width=1em]
                    \tikzstyle{vertex2}=[rectangle, minimum width=1em, minimum height=1em, fill]
        		
        		    \node[draw, vertex] (0) at (-1,0) {$4$};
        		    \node[draw, vertex] (1) at (0,0) {$4$};
        		    \node[draw, vertex] (2) at (1,0) {$4$};
        		    
        		    \draw (0.north) node[above]{$8$};
        		    \draw (1.north) node[above]{$9$};
        		    \draw (2.north) node[above]{$4$};
        		    
        		    \draw (0) -- (1);
        		    \draw[dashed] (1) -- (2);
        		    
                \end{tikzpicture} 
            }
        \caption*{Case 4} 
    \end{subfigure} 
    
    \centering
    
    \begin{subfigure}[b]{0.24\textwidth} 
        \centering 
            \scalebox{1}
	        {
                \begin{tikzpicture}[thick]
                     \tikzstyle{vertex}=[circle,minimum width=1em]
                    \tikzstyle{vertex2}=[rectangle, minimum width=1em, minimum height=1em, fill]
        		
        		    \node[draw, vertex] (0) at (-1,0) {$6$};
        		    \node[draw, vertex] (1) at (0,0) {$6$};
        		    \node[draw, vertex] (2) at (1,0) {$6$};
        		    
        		    \draw (0.north) node[above]{$19$};
        		    \draw (1.north) node[above]{$17$};
        		    \draw (2.north) node[above]{$6$};
        		    
        		    \draw[dashed] (0) -- (1);
        		    \draw[dashed] (1) -- (2);
        		    \draw (0) edge[bend right] (2);
        		    
                \end{tikzpicture} 
            }
        \caption*{Case 5} 
    \end{subfigure} 
    \begin{subfigure}[b]{0.24\textwidth} 
        \centering 
            \scalebox{1}
	        {
                \begin{tikzpicture}[thick]
                     \tikzstyle{vertex}=[circle,minimum width=1em]
                    \tikzstyle{vertex2}=[rectangle, minimum width=1em, minimum height=1em, fill]
        		
        		    \node[draw, vertex] (0) at (-1,0) {$6$};
        		    \node[draw, vertex] (1) at (0,0) {$6$};
        		    \node[draw, vertex] (2) at (1,0) {$6$};
        		    
        		    \draw (0.north) node[above]{$7$};
        		    \draw (1.north) node[above]{$7$};
        		    \draw (2.north) node[above]{$7$};
        		    
        		    \draw[dashed] (0) -- (1);
        		    \draw[dashed] (1) -- (2);
        		    \draw[dashed] (0) edge[bend right] (2);
        		    
                \end{tikzpicture} 
            }
        \caption*{Case 6} 
    \end{subfigure} 
    \begin{subfigure}[b]{0.24\textwidth} 
        \centering 
            \scalebox{1}
	        {
                \begin{tikzpicture}[thick]
                     \tikzstyle{vertex}=[circle,minimum width=1em]
                    \tikzstyle{vertex2}=[rectangle, minimum width=1em, minimum height=1em, fill]
        		
        		    \node[draw, vertex] (0) at (-1,0) {$6$};
        		    \node[draw, vertex] (1) at (0,0) {$6$};
        		    \node[draw, vertex] (2) at (1,0) {$6$};
        		    
        		    \draw (0.north) node[above]{$16$};
        		    \draw (1.north) node[above]{$17$};
        		    \draw (2.north) node[above]{$6$};
        		    
        		    \draw (0) -- (1);
        		    \draw[dashed] (1) -- (2);
        		    
                \end{tikzpicture} 
            }
        \caption*{Case 7} 
    \end{subfigure} 
    \begin{subfigure}[b]{0.24\textwidth} 
        \centering 
            \scalebox{1}
	        {
                \begin{tikzpicture}[thick]
                     \tikzstyle{vertex}=[circle,minimum width=1em]
                    \tikzstyle{vertex2}=[rectangle, minimum width=1em, minimum height=1em, fill]
        		
        		    \node[draw, vertex] (0) at (-1,0) {$4$};
        		    \node[draw, vertex] (1) at (0,0) {$6$};
        		    \node[draw, vertex] (2) at (1,0) {$4$};
        		    
        		    \draw (0.north) node[above]{$19$};
        		    \draw (1.north) node[above]{$17$};
        		    \draw (2.north) node[above]{$19$};
        		    
        		    \draw (0) -- (1);
        		    \draw[dashed] (1) -- (2);
        		    \draw (0) edge[bend right] (2);
        		    
                \end{tikzpicture} 
            }
        \caption*{Case 8} 
    \end{subfigure}
    
    \centering
    
    \begin{subfigure}[b]{0.24\textwidth} 
        \centering 
            \scalebox{1}
	        {
                \begin{tikzpicture}[thick]
                     \tikzstyle{vertex}=[circle,minimum width=1em]
                    \tikzstyle{vertex2}=[rectangle, minimum width=1em, minimum height=1em, fill]
        		
        		    \node[draw, vertex] (0) at (-1,0) {$4$};
        		    \node[draw, vertex] (1) at (0,0) {$6$};
        		    \node[draw, vertex] (2) at (1,0) {$4$};
        		    
        		    \draw (0.north) node[above]{$9$};
        		    \draw (1.north) node[above]{$7$};
        		    \draw (2.north) node[above]{$9$};
        		    
        		    \draw[dashed] (0) -- (1);
        		    \draw[dashed] (1) -- (2);
        		    \draw (0) edge[bend right] (2);
        		    
                \end{tikzpicture} 
            }
        \caption*{Case 9} 
    \end{subfigure}
    \begin{subfigure}[b]{0.24\textwidth} 
        \centering 
            \scalebox{1}
	       {
                \begin{tikzpicture}[thick]
                     \tikzstyle{vertex}=[circle,minimum width=1em]
                    \tikzstyle{vertex2}=[rectangle, minimum width=1em, minimum height=1em, fill]
        		
        		    \node[draw, vertex] (0) at (-1,0) {$4$};
        		    \node[draw, vertex] (1) at (0,0) {$6$};
        		    \node[draw, vertex] (2) at (1,0) {$4$};
        		    
        		    \draw (0.north) node[above]{$8$};
        		    \draw (1.north) node[above]{$17$};
        		    \draw (2.north) node[above]{$4$};
        		    
        		    \draw (0) -- (1);
        		    \draw[dashed] (1) -- (2);
        		    
                \end{tikzpicture} 
            }
        \caption*{Case 10} 
    \end{subfigure}
    \begin{subfigure}[b]{0.24\textwidth} 
        \centering 
            \scalebox{1}
	        {
                \begin{tikzpicture}[thick]
                     \tikzstyle{vertex}=[circle,minimum width=1em]
                    \tikzstyle{vertex2}=[rectangle, minimum width=1em, minimum height=1em, fill]
        		
        		    \node[draw, vertex] (0) at (-1,0) {$4$};
        		    \node[draw, vertex] (1) at (0,0) {$6$};
        		    \node[draw, vertex] (2) at (1,0) {$4$};
        		    
        		    \draw (0.north) node[above]{$4$};
        		    \draw (1.north) node[above]{$7$};
        		    \draw (2.north) node[above]{$4$};
        		    
        		    \draw[dashed] (0) -- (1);
        		    \draw[dashed] (1) -- (2);
        		    
                \end{tikzpicture} 
            }
        \caption*{Case 11} 
    \end{subfigure}
    \begin{subfigure}[b]{0.24\textwidth} 
        \centering 
            \scalebox{1}
	        {
                \begin{tikzpicture}[thick]
                     \tikzstyle{vertex}=[circle,minimum width=1em]
                    \tikzstyle{vertex2}=[rectangle, minimum width=1em, minimum height=1em, fill]
        		
        		    \node[draw, vertex] (0) at (-1,0) {$4$};
        		    \node[draw, vertex] (1) at (0,0) {$4$};
        		    \node[draw, vertex] (2) at (1,0) {$6$};
        		    
        		    \draw (0.north) node[above]{$8$};
        		    \draw (1.north) node[above]{$19$};
        		    \draw (2.north) node[above]{$16$};
        		    
        		    \draw (0) -- (1);
        		    \draw (1) -- (2);
        		    
                \end{tikzpicture} 
            }
        \caption*{Case 12} 
    \end{subfigure}
    
    \centering 
    
    \begin{subfigure}[b]{0.24\textwidth} 
        \centering 
            \scalebox{1}
	        {
                \begin{tikzpicture}[thick]
                     \tikzstyle{vertex}=[circle,minimum width=1em]
                    \tikzstyle{vertex2}=[rectangle, minimum width=1em, minimum height=1em, fill]
        		
        		    \node[draw, vertex] (0) at (-1,0) {$4$};
        		    \node[draw, vertex] (1) at (0,0) {$4$};
        		    \node[draw, vertex] (2) at (1,0) {$6$};
        		    
        		    \draw (0.north) node[above]{$4$};
        		    \draw (1.north) node[above]{$9$};
        		    \draw (2.north) node[above]{$16$};
        		    
        		    \draw[dashed] (0) -- (1);
        		    \draw (1) -- (2);
        		    
                \end{tikzpicture} 
            }
        \caption*{Case 13} 
    \end{subfigure}
    \begin{subfigure}[b]{0.24\textwidth} 
        \centering 
            \scalebox{1}
	        {
                \begin{tikzpicture}[thick]
                     \tikzstyle{vertex}=[circle,minimum width=1em]
                    \tikzstyle{vertex2}=[rectangle, minimum width=1em, minimum height=1em, fill]
        		
        		    \node[draw, vertex] (0) at (-1,0) {$4$};
        		    \node[draw, vertex] (1) at (0,0) {$4$};
        		    \node[draw, vertex] (2) at (1,0) {$6$};
        		    
        		    \draw (0.north) node[above]{$8$};
        		    \draw (1.north) node[above]{$9$};
        		    \draw (2.north) node[above]{$6$};
        		    
        		    \draw (0) -- (1);
        		    \draw[dashed] (1) -- (2);
        		    
                \end{tikzpicture} 
            }
        \caption*{Case 14} 
    \end{subfigure}
    \begin{subfigure}[b]{0.24\textwidth} 
        \centering 
            \scalebox{1}
	        {
                \begin{tikzpicture}[thick]
                     \tikzstyle{vertex}=[circle,minimum width=1em]
                    \tikzstyle{vertex2}=[rectangle, minimum width=1em, minimum height=1em, fill]
        		
        		    \node[draw, vertex] (0) at (-1,0) {$6$};
        		    \node[draw, vertex] (1) at (0,0) {$4$};
        		    \node[draw, vertex] (2) at (1,0) {$6$};
        		    
        		    \draw (0.north) node[above]{$17$};
        		    \draw (1.north) node[above]{$19$};
        		    \draw (2.north) node[above]{$17$};
        		    
        		    \draw (0) -- (1);
        		    \draw (1) -- (2);
        		    \draw (0) edge[bend right, dashed] (2);
        		    
                \end{tikzpicture} 
            }
        \caption*{Case 15} 
    \end{subfigure}
    \begin{subfigure}[b]{0.24\textwidth} 
        \centering 
            \scalebox{1}
	        {
                \begin{tikzpicture}[thick]
                     \tikzstyle{vertex}=[circle,minimum width=1em]
                    \tikzstyle{vertex2}=[rectangle, minimum width=1em, minimum height=1em, fill]
        		
        		    \node[draw, vertex] (0) at (-1,0) {$6$};
        		    \node[draw, vertex] (1) at (0,0) {$4$};
        		    \node[draw, vertex] (2) at (1,0) {$6$};
        		    
        		    \draw (0.north) node[above]{$7$};
        		    \draw (1.north) node[above]{$9$};
        		    \draw (2.north) node[above]{$19$};
        		    
        		    \draw[dashed] (0) -- (1);
        		    \draw (1) -- (2);
        		    \draw (0) edge[bend right, dashed] (2);
        		    
                \end{tikzpicture} 
            }
        \caption*{Case 16} 
    \end{subfigure}
    
    \centering
    
    \begin{subfigure}[b]{0.24\textwidth} 
        \centering 
            \scalebox{1}
	        {
                \begin{tikzpicture}[thick]
                     \tikzstyle{vertex}=[circle,minimum width=1em]
                    \tikzstyle{vertex2}=[rectangle, minimum width=1em, minimum height=1em, fill]
        		
        		    \node[draw, vertex] (0) at (-1,0) {$6$};
        		    \node[draw, vertex] (1) at (0,0) {$4$};
        		    \node[draw, vertex] (2) at (1,0) {$6$};
        		    
        		    \draw (0.north) node[above]{$16$};
        		    \draw (1.north) node[above]{$19$};
        		    \draw (2.north) node[above]{$16$};
        		    
        		    \draw (0) -- (1);
        		    \draw (1) -- (2);
        		    
                \end{tikzpicture} 
            }
        \caption*{Case 17} 
    \end{subfigure}
    \begin{subfigure}[b]{0.24\textwidth} 
        \centering 
            \scalebox{1}
	        {
                \begin{tikzpicture}[thick]
                     \tikzstyle{vertex}=[circle,minimum width=1em]
                    \tikzstyle{vertex2}=[rectangle, minimum width=1em, minimum height=1em, fill]
        		
        		    \node[draw, vertex] (0) at (-1,0) {$6$};
        		    \node[draw, vertex] (1) at (0,0) {$4$};
        		    \node[draw, vertex] (2) at (1,0) {$6$};
        		    
        		    \draw (0.north) node[above]{$6$};
        		    \draw (1.north) node[above]{$9$};
        		    \draw (2.north) node[above]{$16$};
        		    
        		    \draw[dashed] (0) -- (1);
        		    \draw (1) -- (2);
        		    
                \end{tikzpicture} 
            }
        \caption*{Case 18} 
    \end{subfigure}
    \begin{subfigure}[b]{0.24\textwidth} 
        \centering 
            \scalebox{1}
	        {
                \begin{tikzpicture}[thick]
                     \tikzstyle{vertex}=[circle,minimum width=1em]
                    \tikzstyle{vertex2}=[rectangle, minimum width=1em, minimum height=1em, fill]
        		
        		    \node[draw, vertex] (0) at (-1,0) {$6$};
        		    \node[draw, vertex] (1) at (0,0) {$6$};
        		    \node[draw, vertex] (2) at (1,0) {$4$};
        		    
        		    \draw (0.north) node[above]{$16$};
        		    \draw (1.north) node[above]{$17$};
        		    \draw (2.north) node[above]{$4$};
        		    
        		    \draw[dashed] (0) -- (1);
        		    \draw (1) -- (2);
        		    
                \end{tikzpicture} 
            }
        \caption*{Case 19} 
    \end{subfigure}
    \begin{subfigure}[b]{0.24\textwidth} 
        \centering 
            \scalebox{1}
	        {
                \begin{tikzpicture}[thick]
                     \tikzstyle{vertex}=[circle,minimum width=1em]
                    \tikzstyle{vertex2}=[rectangle, minimum width=1em, minimum height=1em, fill]
        		
        		    \node[draw, vertex] (0) at (-1,0) {$6$};
        		    \node[draw, vertex] (1) at (0,0) {$6$};
        		    \node[draw, vertex] (2) at (1,0) {$4$};
        		    
        		    \draw (0.north) node[above]{$6$};
        		    \draw (1.north) node[above]{$17$};
        		    \draw (2.north) node[above]{$8$};
        		    
        		    \draw (0) -- (1);
        		    \draw[dashed] (1) -- (2);
        		    
                \end{tikzpicture} 
            }
        \caption*{Case 20} 
    \end{subfigure}
\end{figure}
\begin{itemize}
\item Case 1: We can use Proposition~\ref{prop:19-5-4} to show that the three vertices forbid at most 17 colors from $u$.
\item Case 2: We can use Proposition~\ref{prop:19-5-4} since having one less edge gives us less constraints.
\item Case 3: Notice that it is more restrictive for two vertices $u$ and 
$v$ to be adjacent rather than to be adjacent to the same 2-vertex. This comes from the fact that the 2-vertex can be colored using Property $P_{2, 4}$ as long as $u$ and $v$ get different identities (which is already a 
requirement when $u$ and $v$ are adjacent). We can therefore use Proposition~\ref{prop:19-5-4}. If there are more that one 2-vertex that $u$ and $v$ are both adjacent to we can still use this technique (and $u$ and $v$ have more available colors). In the following cases we can therefore assume that there is at most one $2$-vertex adjacent to a given pair of bad vertices.
\item Case 4: Note that the colors in $S_2$ all have different identities 
since they are all adjacent to the same vertex in $\rho(SP_9^+)$. By removing at most 4 of the colors from $S_2$, we can guarantee that any color chosen for $v_2$ will not have the same identity as one of the colors available for $v_3$. By Proposition~\ref{prop:7-5}, $v_1$ and $v_2$ forbid at most 13 colors and by Proposition~\ref{prop:P24*} $v_3$ forbids at most 
2 colors.
\item Case 5: We use Proposition~\ref{prop:19-5-4}.
\item Case 6: By removing at most $4$ colors from $S_1$ and $2$ colors from $S_2$ and $S_3$ we can guarantee that no colors in these three sets have the same identity. By Proposition~\ref{prop:cf}, the three vertices forbid at most $6 + 6 + 4 = 16$.
\item Case 7: We remove $3$ colors from $S_3$. By removing at most $9$ colors from $S_2$ we can guarantee that no color in $S_2$ has the same identity as a color in $S_3$. By Proposition~\ref{prop:7-5}, $v_1$ and $v_2$ forbid at most 13 colors and by Proposition~\ref{prop:cf} $v_3$ forbids at most 4 colors. 
\item Case 8: We use Proposition~\ref{prop:19-5-4}.
\item Case 9: Suppose that the $7$ identities in $S_2$ are all in $S_1$ and $S_3$. We remove $2$ colors from $v_1$ and $v_3$ and $5$ colors from $v_2$ such that $S_1 = S_3$ and no color in $S_1$ has the same identity as a color in $S_2$. Vertices $v_1$ and $v_2$ forbid at most 11 colors by 
Proposition~\ref{prop:7-7} and by Proposition~\ref{prop:cf} $v_3$ forbids 
at most 6 colors. 
We can now assume that there is at least one identity in $S_2$ that is not in $S_1$ (or alternatively $S_3$). Therefore, by removing at most $2$ vertices from $S_1$, $4$ from $S_2$ and $3$ from $S_3$ we can guarantee that there are no colors in $S_2$ that have the same identity as a color in 
$S_1$ or $S_3$. By Proposition~\ref{prop:7-5}, $v_1$ and $v_3$ forbid at most 13 colors and by Proposition~\ref{prop:cf} $v_2$ forbids at most 4 colors. 
\item Case 10: We proceed similarly to Case $7$.
\item Case 11: We remove $1$ color from $S_1$ and $S_3$ and $6$ from $S_2$ such that there are no colors in $S_2$ that have the same identity as a 
color in $S_1$ or $S_3$. By Property $P_{1, 9}$ of $\rho(SP_9^+)$, $v_2$ forbids at most $11$ colors and by Property~\ref{prop:P24*} $v_1$ and $v_3$ each forbid at most $3$ colors.
\item Case 12: We use Proposition~\ref{prop:19-5-4}.
\item Case 13: We proceed similarly to Case $4$.
\item Case 14: We proceed similarly to Case $4$.
\item Case 15: We use Proposition~\ref{prop:19-5-4}. 
\item Case 16: We use Proposition~\ref{prop:19-5-4}.
\item Case 17: We use Proposition~\ref{prop:19-5-4}.
\item Case 18: We proceed similarly to Case $5$.
\item Case 19: We proceed similarly to Case $5$.
\item Case 20: We proceed similarly to Case $5$.
\end{itemize}
\end{itemize}
\end{proof}
\begin{lemma}
The graph $G$ does not contain Configurations $C_{\ref{SPnpC7}}$ or $C_{\ref{SPnpC8}}$.
\end{lemma}

\begin{proof}
We proceed in the same way as Configuration $C_{\ref{SPnpC6}}$ except there are at the start only $9$ (resp. $4$) colors available for $u$ by Property $P_{1, 9}$ (resp. $P_{2, 4}$) of $\rho(SP_9^+)$.
\end{proof}

\subsection{Discharging}

We start by the definition of the initial weighting $\omega$ defined
by $\omega(v)=d(v)-\frac{17}{5}$ for each vertex $v$ of degree $d(v)$. By
construction, the sum of all the weights is negative.

We then introduce two discharging rules:
\begin{itemize}
\item[($\ruleSPnp{R2}$)] Every $4^+$-vertex gives $\frac{7}{10}$ to each of its 2-neighbors.
\item[($\ruleSPnp{Rb}$)] Every $4^+$ good vertex gives $\frac{1}{10}$ to each of its bad neighbors.
\end{itemize}

\label{sec:SP9+positive}
This section is devoted to obtaining a contradiction by proving that
every vertex of $G$ has non-negative final weights after the
discharging procedure. We distinguish several cases for the vertices,
depending on their degree. First note that since $G$ cannot contain $C_{\ref{SPnpC0}}$ and $C_{\ref{SPnpC1}}$, the minimum degree of $G$ is $2$ and $G$ does not contain 3-vertices by $C_{\ref{SPnpC5}}$.

\subsubsection{$2$-vertices}

Let $v$ be a $2$-vertex. Since $C_{\ref{SPnpC2}}$ and $C_{\ref{SPnpC5}}$ are forbidden, $v$ only has
$4^+$-neighbors. Thus, by $R_{\ref{R2}}$, each of them gives
$\frac{7}{10}$ to $v$. Therefore, the final weight of $v$ is
$\omega'(v)=2-\frac{17}{5}+2\cdot \frac{7}{10}=0$.

\subsubsection{vertices of degree $d$, $4 \leq d \leq 11$}
We checked on a computer with the following algorithm that for every vertex $v$ with $b$ bad neighbours, $t$ 2-neighbors and $n$ other neighbors then either $v$ is in a forbidden configurations or $v$ has final weight at least 0 after discharging. \newline

\begin{algorithm}[H]
 \For{$degree \in \{4, 5, ..., 11\}$}{
     \For{$t, b, n \in \mathbb{N}$ such that $t + b + n = degree$}{
            \uIf{$n = 0$ and $e + 4 \cdot b < 20$ and $b \leq 2$}{
                continue (forbidden configuration $C_{\ref{SPnpC6}}$)
            }
            \uElseIf{$n = 0$ and $t + 17 < 20$ and $b = 3$}{
                continue (forbidden configuration $C_{\ref{SPnpC6}}$)
            }
            \uElseIf{$n = 1$ and $t + 4 \cdot b < 9$ and $b \leq 2$}{
                continue (forbidden configuration $C_{\ref{SPnpC7}}$)
            }
            \uElseIf{$n = 2$ and $t < 4$ and $b = 0$}{
                continue (forbidden configuration $C_{\ref{SPnpC8}}$)
            }
            \uElseIf{$v$ is bad}{
                \uIf{$degree - \frac{17}{5} + (n) * \frac{1}{10} - e * \frac{7}{10}$}{
                continue (final weight at least 0)
                }
                \Else{
                error (final weight smaller than 0)
                }
            }
            \Else{
                \uIf{$degree - \frac{17}{5} - t * \frac{7}{10} - b * \frac{1}{10} \geq 0$}{
                continue (final weight at least 0)
                }
                \Else{
                error (final weight smaller than 0)
                }
            }
     }
 }
 \caption*{Algorithm used to check that each vertex of degree between 4 and 11 has final weight at least 0 after discharging.}
\end{algorithm}

\subsubsection{$12^+$-vertices}
Let $v$ be a vertex of degree $d$ at least 12. In the worse case, $v$ has 
$d$ 2-neighbors. Therefore, it has weight at least $d - \frac{17}{5} - d \cdot \frac{7}{10}$ which is greater than or equal to 0 for $d \geq 12$.

Every vertex has non-negative weight after discharging so $G$ cannot have 
maximum average degree smaller than $\frac{17}{5}$. This gives us a contradiction and concludes the proof.

\section{Proof of Theorem \ref{thm:SPq+}}\label{sec:thSPq+}
In this section, for $q\ge 9$, we prove that any signed graph of maximum average degree less than $4 - \frac{8}{q+3}$ admit a $\rho(SP_q^+)$-sp-coloring. To do so, we suppose that this theorem is false and we consider in the remainder of this section a minimal counter-example G w.r.t its order: it is a smallest signed graph with $\mad(G) <4- \frac{8}{q+3}$ admitting no $\rho(SP_q^+)$-sp-coloring.

\subsection{Forbidden configurations}
\label{sec:SPq+reduction}\label{sec:SPq+configs}

We define several configurations $C_1,\ldots,C_{9}$ as follows (see Figure~\ref{fig:SPq+configs}).

\begin{itemize}
\item[$\bullet$] $\configSPqp{C0}$ is a $0$-vertex.
\item[$\bullet$] $\configSPqp{C1}$ is a $1$-vertex.
\item[$\bullet$] $\configSPqp{C2}$ is two adjacent 2-vertices.
\item[$\bullet$] $\configSPqp{C3}$ is a 2-vertex with a 3-neighbor.
\item[$\bullet$] $\configSPqp{C4}$ is a 3-vertex.
\item[$\bullet$] $\configSPqp{C5}$ is a $d$-vertex adjacent to at most $d$ 2-neighbors with $d < 2q + 2$.
\item[$\bullet$] $\configSPqp{C6}$ is a $d$-vertex adjacent to at most $d-1$ 2-neighbors with $d < q+1$.
\item[$\bullet$] $\configSPqp{C7}$ is a $d$-vertex adjacent to at most $d-2$ 2-neighbors with $d < \frac{q+3}{2}$.
\item[$\bullet$] $\configSPqp{C8}$ is a $d$-vertex adjacent to at most $d-3$ 2-neighbors with $d < \frac{q+7}{4}$.
\end{itemize}

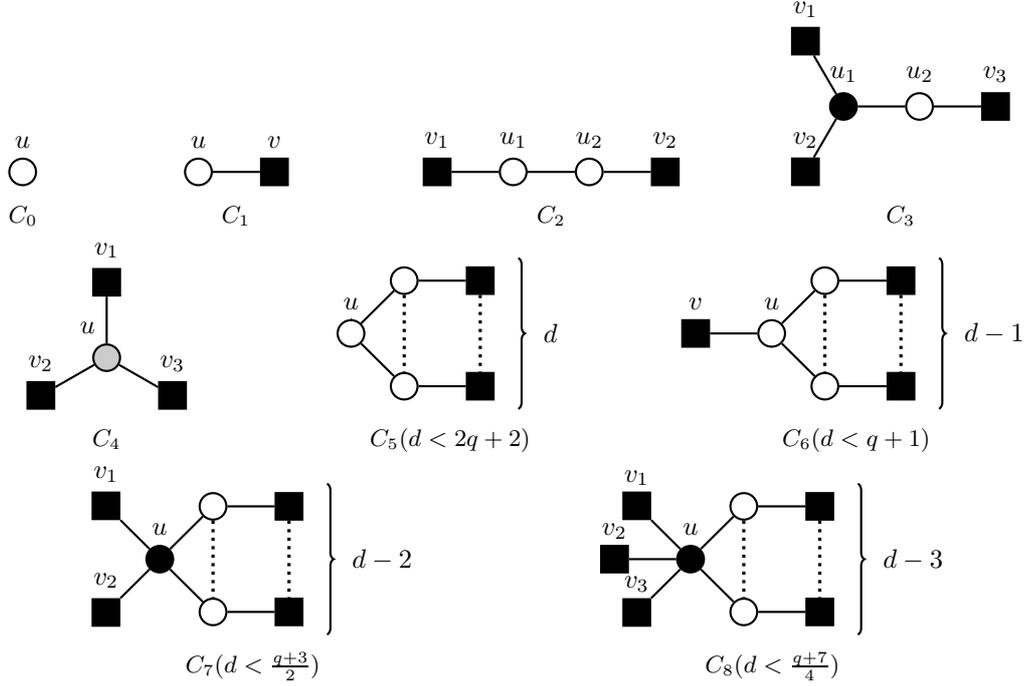
\begin{figure}[H]
    \centering
    \begin{subfigure}[b]{0.12\textwidth}
        \centering 
            \scalebox{1}
	        {
                \begin{tikzpicture}[thick]
                    \tikzstyle{vertex}=[circle,minimum width=1em]
                    \tikzstyle{vertex2}=[rectangle, minimum width=1em, minimum height=1em]
        		
        		    \node[draw, vertex] (1) at (0,0) {};
        		    \draw (1.north) node[above]{$u$};
        		    
                \end{tikzpicture} 
            }
        \caption*{$C_0$} 
    \end{subfigure} 
    \begin{subfigure}[b]{0.24\textwidth}
        \centering 
            \scalebox{1}
	        {
                \begin{tikzpicture}[thick]
                    \tikzstyle{vertex}=[circle,minimum width=1em]
                    \tikzstyle{vertex2}=[rectangle, minimum width=1em, minimum height=1em]
        		
        		    \node[draw, vertex] (1) at (0,0) {};
        		    \node[draw, vertex2, fill] (2) at (1,0) {};
        		    
        		    \draw (1.north) node[above]{$u$};
        		    \draw (2.north) node[above]{$v$};
        		    
        		    \draw (1) -- (2);
                \end{tikzpicture} 
            }
        \caption*{$C_1$} 
    \end{subfigure} 
    \begin{subfigure}[b]{0.30\textwidth} 
        \centering 
            \scalebox{1}
	        {
                \begin{tikzpicture}[thick]
                    \tikzstyle{vertex}=[circle,minimum width=1em]
                    \tikzstyle{vertex2}=[rectangle, minimum width=1em, minimum height=1em, fill]
        		
        		    \node[draw, vertex] (1) at (0,0) {};
        		    \node[draw, vertex] (2) at (1,0) {};
        		    \node[draw, vertex2] (0) at (-1,0) {};
        		    \node[draw, vertex2] (3) at (2,0) {};
        		    
        		    \draw (0.north) node[above]{$v_1$};
        		    \draw (1.north) node[above]{$u_1$};
        		    \draw (2.north) node[above]{$u_2$};
        		    \draw (3.north) node[above]{$v_2$};
        		    
        		    \draw (0) -- (1);
        		    \draw (1) -- (2);
        		    \draw (2) -- (3);
                \end{tikzpicture} 
            }
        \caption*{$C_2$} 
    \end{subfigure} 
    \begin{subfigure}[b]{0.30\textwidth} 
        \centering 
            \scalebox{1}
	        {
                \begin{tikzpicture}[thick]
                    \tikzstyle{vertex}=[circle,minimum width=1em]
                    \tikzstyle{vertex2}=[rectangle, minimum width=1em, minimum height=1em, fill]
        		
        		    \node[draw, vertex, fill] (0) at (0:0) {};
        		    \node[draw, vertex] (1) at (0:1) {};
        		    \node[draw, vertex2] (2) at (0:2) {};
        		    \node[draw, vertex2] (3) at (120:1) {};
        		    \node[draw, vertex2] (4) at (240:1) {};
        		    
        		    \draw (0.north) node[above]{$u_1$};
        		    \draw (1.north) node[above]{$u_2$};
        		    \draw (2.north) node[above]{$v_3$};
        		    \draw (3.north) node[above]{$v_1$};
        		    \draw (4.north) node[above]{$v_2$};
        		    
        		    \draw (0) -- (1);
        		    \draw (1) -- (2);
        		    \draw (0) -- (3);
        		    \draw (0) -- (4);
                \end{tikzpicture}  
            }
        \caption*{$C_3$} 
    \end{subfigure} 
    
    \centering
        
    \begin{subfigure}[b]{0.24\textwidth} 
        \centering 
            \scalebox{1}
	        {
                \begin{tikzpicture}[thick]
                    \tikzstyle{vertex}=[circle,minimum width=1em]
                    \tikzstyle{vertex2}=[rectangle, minimum width=1em, minimum height=1em, fill]

        		    \node[draw, vertex, fill=gray!40] (0) at (0:0) {};
        		    \node[draw, vertex2] (1) at (0-30:1) {};
        		    \node[draw, vertex2] (3) at (120-30:1) {};
        		    \node[draw, vertex2] (4) at (240-30:1) {};
        		    
        		    \draw (0.north) node[above, shift={(-0.25,0)}]{$u$};
        		    \draw (1.north) node[above]{$v_3$};
        		    \draw (3.north) node[above]{$v_1$};
        		    \draw (4.north) node[above]{$v_2$};
        		    
        		    \draw (0) -- (1);
        		    \draw (0) -- (3);
        		    \draw (0) -- (4);
                \end{tikzpicture} 
            }
        \caption*{$C_4$} 
    \end{subfigure} 
    \begin{subfigure}[b]{0.35\textwidth} 
        \centering 
            \scalebox{1}
	        {
                \begin{tikzpicture}[thick]
                    \tikzstyle{vertex}=[circle,minimum width=1em]
                    \tikzstyle{vertex2}=[rectangle, minimum width=1em, minimum height=1em, fill]
                    \tikzstyle{vertex3}=[regular polygon, regular polygon sides=3, minimum size=1em]
        		
        		    \node[draw, vertex] (1) at (0,0) {};
        		    \node[draw, vertex] (2) at (0.7, 0.7) {};
        		    \node[draw, vertex2] (2') at (1.7, 0.7) {};
        		    \node[draw, vertex] (3) at (0.7, -0.7) {};
        		    \node[draw, vertex2] (3') at (1.7, -0.7) {};
        		    
        		    \draw (1.north) node[above]{$u$};
        		    
        		    \draw (0) -- (1);
        		    \draw (1) -- (2);
        		    \draw (1) -- (3);
        		    \draw (2) -- (2');
        		    \draw (3) -- (3');
        		    \draw[dotted, very thick] (2) -- (3);
        		    \draw[dotted, very thick] (2') -- (3');
        		    
        		    \draw[decorate,decoration={brace}](2.2,1) -- (2.2,-1) node[right=0.2cm,pos=0.5] {$d$};
                \end{tikzpicture} 
            }
        \caption*{$C_5 (d < 2q+2)$} 
    \end{subfigure} 
    \begin{subfigure}[b]{0.35\textwidth} 
        \centering 
            \scalebox{1}
	        {
                \begin{tikzpicture}[thick]
                    \tikzstyle{vertex}=[circle,minimum width=1em]
                    \tikzstyle{vertex2}=[rectangle, minimum width=1em, minimum height=1em, fill]
                    \tikzstyle{vertex3}=[regular polygon, regular polygon sides=3, minimum size=1em]
        		
        		    \node[draw, vertex] (1) at (0,0) {};
        		    \node[draw, vertex] (2) at (0.7, 0.7) {};
        		    \node[draw, vertex2] (2') at (1.7, 0.7) {};
        		    \node[draw, vertex] (3) at (0.7, -0.7) {};
        		    \node[draw, vertex2] (3') at (1.7, -0.7) {};
        		    \node[draw, vertex2] (0) at (-1,0) {};
        		    
        		    \draw (1.north) node[above]{$u$};
        		    \draw (0.north) node[above]{$v$};

        		    \draw (0) -- (1);
        		    \draw (1) -- (2);
        		    \draw (1) -- (3);
        		    \draw (2) -- (2');
        		    \draw (3) -- (3');
        		    \draw[dotted, very thick] (2) -- (3);
        		    \draw[dotted, very thick] (2') -- (3');
        		    
        		    \draw[decorate,decoration={brace}](2.2,1) -- (2.2,-1) node[right=0.2cm,pos=0.5] {$d-1$};
                \end{tikzpicture} 
            }
        \caption*{$C_6 (d < q+1)$} 
    \end{subfigure} 
    
    \centering
    
    \begin{subfigure}[b]{0.45\textwidth} 
        \centering 
            \scalebox{1}
	        {
                \begin{tikzpicture}[thick]
                    \tikzstyle{vertex}=[circle,minimum width=1em]
                    \tikzstyle{vertex2}=[rectangle, minimum width=1em, minimum height=1em, fill]
                    \tikzstyle{vertex3}=[regular polygon, regular polygon sides=3, minimum size=1em]
        		
        		    \node[draw, vertex, fill] (1) at (0,0) {};
        		    \node[draw, vertex] (2) at (0.7, 0.7) {};
        		    \node[draw, vertex2] (2') at (1.7, 0.7) {};
        		    \node[draw, vertex] (3) at (0.7, -0.7) {};
        		    \node[draw, vertex2] (3') at (1.7, -0.7) {};
        		    \node[draw, vertex2] (0) at (135:1) {};
        		    \node[draw, vertex2] (0') at (-135:1) {};
        		    
        		    \draw (1.north) node[above]{$u$};
        		    \draw (0.north) node[above]{$v_1$};
        		    \draw (0'.north) node[above]{$v_2$};

        		    \draw (0) -- (1);
        		    \draw (0') -- (1);
        		    \draw (1) -- (2);
        		    \draw (1) -- (3);
        		    \draw (2) -- (2');
        		    \draw (3) -- (3');
        		    \draw[dotted, very thick] (2) -- (3);
        		    \draw[dotted, very thick] (2') -- (3');
        		    
        		    \draw[decorate,decoration={brace}](2.2,1) -- (2.2,-1) node[right=0.2cm,pos=0.5] {$d-2$};
                \end{tikzpicture} 
            }
        \caption*{$C_7 (d < \frac{q+3}{2})$} 
    \end{subfigure} 
    \begin{subfigure}[b]{0.45\textwidth} 
        \centering 
            \scalebox{1}
	        {
                \begin{tikzpicture}[thick]
                    \tikzstyle{vertex}=[circle,minimum width=1em]
                    \tikzstyle{vertex2}=[rectangle, minimum width=1em, minimum height=1em, fill]
                    \tikzstyle{vertex3}=[regular polygon, regular polygon sides=3, minimum size=1em]
        		
        		    \node[draw, vertex, fill] (1) at (0,0) {};
        		    \node[draw, vertex] (2) at (0.7, 0.7) {};
        		    \node[draw, vertex2] (2') at (1.7, 0.7) {};
        		    \node[draw, vertex] (3) at (0.7, -0.7) {};
        		    \node[draw, vertex2] (3') at (1.7, -0.7) {};
        		    \node[draw, vertex2] (0) at (-1,0) {};
        		    \node[draw, vertex2] (0') at (135:1) {};
        		    \node[draw, vertex2] (0'') at (-135:1) {};
        		    
        		    \draw (1.north) node[above]{$u$};
        		    \draw (0'.north) node[above, shift={(-0,-0.05)}]{$v_1$};
        		    \draw (0.north) node[above, shift={(-0,-0.05)}]{$v_2$};
        		    \draw (0''.north) node[above, shift={(-0,-0.05)}]{$v_3$};
        		    
        		    \draw (0) -- (1);
        		    \draw (0') -- (1);
        		    \draw (0'') -- (1);
        		    \draw (1) -- (2);
        		    \draw (1) -- (3);
        		    \draw (2) -- (2');
        		    \draw (3) -- (3');
        		    \draw[dotted, very thick] (2) -- (3);
        		    \draw[dotted, very thick] (2') -- (3');
        		    
        		    \draw[decorate,decoration={brace}](2.2,1) -- (2.2,-1) node[right=0.2cm,pos=0.5] {$d-3$};
                \end{tikzpicture}  
            }
        \caption*{$C_8 (d < \frac{q+7}{4})$} 
    \end{subfigure}
    \caption{Forbidden configurations. Square vertices can be of any degree. White vertices will be removed. Triangle vertices are 2, 5-worse or 6-bad-vertices.} 
    \label{fig:SPq+configs}
  \end{figure}
  
We prove that configurations $C_0$ to
$C_8$ are forbidden. To this end, we first prove some
generic results we use to prove that the configurations are forbidden. Remember that $\rho(SP_q^+)$ is vertex-transitive, antiautomorphic and has Properties $P_{1,q}$, $P_{2, \frac{q-1}{2}}$ and $P_{3, \frac{q-5}{4}}$ by Lemma~\ref{lem:PTRSP}.

\begin{lemma}
The graph $G$ does not contain Configuration $C_0$ to $C_4$.
\end{lemma}

\begin{proof}
The target graph $\rho(SP_q^+)$ is vertex-transitive, antiautomorphic and 
has Properties $P_{1,q}$, $P_{2, \frac{q-1}{2}}$ and $P_{3, \frac{q-5}{4}}$ by Lemma~\ref{lem:PTRSP}. Since $q\ge 9$, it has at least Properties $P_{1,9}$, $P_{2, 4}$ and $P_{3,1}$ which are the same properties has $\rho(SP_9^+)$ (see Lemma~\ref{lem:PrhoSP}). Therefore, Configurations $C_0$ to $C_4$ are forbidden by Lemmas~\ref{TRSP9_C0},~\ref{TRSP9_C1},~\ref{TRSP9_C2}, \ref{TRSP9_C3}, and~\ref{lem:TRSP9_C4}.
\end{proof}

Since $\rho(SP_q^+)$ has better properties than $\rho(SP_5^+)$ since $q\ge 9$, we will apply Claim~\ref{c:SP5-2} in the proofs of this section. 

\begin{lemma}
The graph $G$ does not contain Configuration $C_5$.
\end{lemma}

\begin{proof}
Suppose that $G$ contains configuration $C_5$. By minimality of $G$, the graph obtained from $G$ by removing $u$ and its $2$-neighbors admits a $\rho(SP_q^+)$-sp-coloring $\varphi$. Every $2$-neighbor of $u$ forbids at most 1 color from $u$ by Claim~\ref{c:SP5-2}. Since there are $2q+2$ colors in $G$, we can find always find a color for $u$ to extend $\varphi$ to 
a $\rho(SP_q^+)$-sp-coloring of $G$, a contradiction.
\end{proof}

\begin{lemma}
The graph $G$ does not contain Configuration $C_6$.
\end{lemma}

\begin{proof}
Suppose that $G$ contains configuration $C_6$. By minimality of $G$, the graph obtained from $G$ by removing $u$ and its $2$-neighbors admits a $\rho(SP_q^+)$-sp-coloring $\varphi$. By Property $P_{1, q}$, we have $q$ available colors for $u$. Every $2$-neighbor of $u$ forbids at most 1 color from $u$ by Claim~\ref{c:SP5-2}. We can therefore always find a color for $u$ to extend $\varphi$ to a $\rho(SP_q^+)$-sp-coloring of $G$, a contradiction.
\end{proof}

\begin{lemma}
The graph $G$ does not contain Configuration $C_7$.
\end{lemma}

\begin{proof}
Suppose that $G$ contains configuration $C_7$. By minimality of $G$, the graph obtained from $G$ by removing the $2$-neighbors of $u$ admits a $\rho(SP_q^+)$-sp-coloring $\varphi$. By Property $P_{2, \frac{q-1}{2}}$, $u$ can be recolored in $\frac{q-1}{2}$ distinct colors such that there is no conflict with $\varphi(v_1)$ and $\varphi(v_2)$. Every $2$-neighbor of 
$u$ forbids at most 1 color from $u$ by Claim~\ref{c:SP5-2}. We can therefore always find a color for $u$ to extend $\varphi$ to a $\rho(SP_q^+)$-sp-coloring of $G$, a contradiction.
\end{proof}

\begin{lemma}
The graph $G$ does not contain Configuration $C_8$.
\end{lemma}

\begin{proof}
Suppose that $G$ contains configuration $C_8$. By minimality of $G$, the graph obtained from $G$ by removing the $2$-neighbors of $u$ admits a $\rho(SP_q^+)$-coloring $\varphi$. By Property $P_{3, \frac{q-5}{4}}$, $u$ can be recolored in $\frac{q-5}{4}$ colors such that there is no conflict with $\varphi(v_1)$, $\varphi(v_2)$ and $\varphi(v_3)$. Every $2$-neighbor of $u$ forbids at most 1 color from $u$  by Claim~\ref{c:SP5-2}. We can 
therefore always find a color for $u$ to extend $\varphi$ to a $\rho(SP_q^+)$-sp-coloring of $G$, a contradiction.
\end{proof}

\subsection{Discharging}
\label{sec:SPq+positive}
Let $\omega$ be the initial weighting defined
by $\omega(v)=d(v)-4+\frac{8}{q+3}$ for each vertex $v$ of degree $d(v)$. By
construction, the sum of all the weights is negative since $\mad(G) < 4-\frac{8}{q+3}$.

We introduce the following discharging rule:
\begin{itemize}
\item[($R$)] Every $4^+$-vertex gives $\frac{q-1}{q+3}$ to each of its neighbors of degree 2.
\end{itemize}

This section is devoted to obtaining a contradiction by proving that
every vertex of $G$ has non-negative final weights after the
discharging procedure. We distinguish several cases for the vertices,
depending on their degree. First note that since $G$ cannot contain $C_0$, $C_1$ and $C_4$, $G$ contains no 0, 1 or 3-vertex.

\paragraph{$2$-vertices.}

Let $v$ be a $2$-vertex. Since $C_2$ is forbidden, $v$ only has
$4^+$-neighbors. Thus, each of them gives $\frac{q-1}{q+3}$ to $v$. Therefore, the final weight of $v$ is $2 - \left(4 - \frac{8}{q+3}\right) + 2\cdot \frac{q+1}{q+3}= 0$.

\paragraph{$d$-vertices with $4 \leq d < \frac{q+7}{4}$.}

Let $v$ be such a $d$-vertex. Since $C_8$ is forbidden, $v$ has at most $d-4$ $2$-neighbors. Therefore, in the worst case, $v$ has final weight at 
least $d - \left(4-\frac{8}{q+3}\right) - (d-4) \cdot \frac{q-1}{q+3} \geq 4 - \left(4-\frac{8}{q+3}\right) > 0$.

\paragraph{$d$-vertices with $\frac{q+7}{4} \leq d < \frac{q+3}{2}$.}

Let $v$ be such a $d$-vertex. Since $C_7$ is forbidden, $v$ has at most $d-3$ $2$-neighbors. Therefore, in the worst case, $v$ has final weight at 
least $d - \left(4-\frac{8}{q+3}\right) - (d-3) \cdot \frac{q-1}{q+3} \geq \frac{q+7}{4} - \left(4-\frac{8}{q+3}\right) - \left(\frac{q+7}{4}-3\right) \cdot \frac{q-1}{q+3} = 0$.

\paragraph{$d$-vertices with $\frac{q+3}{2} \leq d < q+1$.}

Let $v$ be such a $d$-vertex. Since $C_6$ is forbidden, $v$ has at most $d-2$ $2$-neighbors. Therefore, in the worst case, $v$ has final weight at 
least $d - \left(4-\frac{8}{q+3}\right) - (d-2) \cdot \frac{q-1}{q+3} \geq \frac{q+3}{2} - \left(4-\frac{8}{q+3}\right) - \left(\frac{q+3}{2}-2\right) \cdot \frac{q-1}{q+3} = 0$.

\paragraph{$d$-vertices with $q+1 \leq d < 2q+2$.}

Let $v$ be such a $d$-vertex. Since $C_5$ is forbidden, $v$ has at most $d-1$ $2$-neighbors. Therefore, in the worst case, $v$ has final weight at 
least $d - \left(4-\frac{8}{q+3}\right) - (d-1) \cdot \frac{q-1}{q+3} \geq q+1 - \left(4-\frac{8}{q+3}\right) - (q+1-1) \cdot \frac{q-1}{q+3} > 0$.

\paragraph{$d$-vertices with $2q+2 \leq d$.}

Let $v$ be such a $d$-vertex. Vertex $v$ has at most $d$ $2$-neighbors. Therefore, in the worst case, $v$ has final weight at least $d - \left(4-\frac{8}{q+3}\right) - d \cdot \frac{q-1}{q+3} \geq 2q+2 - \left(4-\frac{8}{q+3}\right) - (2q+2) \cdot \frac{q-1}{q+3} > 0$.\newline \newline

Every vertex has non-negative weight after discharging so $G$ cannot have 
maximum average degree smaller than $4-\frac{8}{q+3}$. This gives us a contradiction and concludes the proof.

\section*{Acknowledgments}
The authors would like to thank F. Dross, M. Montassier, T. Pierron, M. Senhaji, A. Raspaud and Z. Wang for fruitful conversations at the 2019 HoSiGra meeting.

\bibliographystyle{plain}

\end{document}